\documentclass[12pt,reqno,a4paper]{amsart}

\usepackage{mdframed}
\usepackage{setspace}
\usepackage{mathrsfs}  
\usepackage{comment}

\usepackage[latin1]{inputenc}
\usepackage[linkcolor=red,colorlinks=true]{hyperref}
\usepackage{bold-extra}
\usepackage{etoolbox}
\usepackage{amsmath}
\usepackage{enumerate}
\usepackage{amssymb}
\usepackage{amscd}
\usepackage{amsthm}
\usepackage{amsfonts}
\usepackage{graphicx}
\usepackage[all,cmtip]{xy}

\patchcmd{\subsection}{-.5em}{.5em}{}{}
\patchcmd{\subsubsection}{-.5em}{.5em}{}{}

\makeatother
\usepackage{hyperref}

\numberwithin{equation}{section}

\newcommand{\cO}{\mathcal{O}}

\newcommand{\cT}{\mathcal{T}}

\newcommand{\bC}{\mathbb{C}}

\newcommand{\bH}{\mathbb{H}}

\newcommand{\bQ}{\mathbb{Q}}
\newcommand{\bR}{\mathbb{R}}

\newcommand{\bZ}{\mathbb{Z}}
\newcommand{\ra}{\rightarrow}

\newcommand{\qand}{\quad \textrm{and} \quad}

\newcommand\subsetsim{\mathrel{%
\ooalign{\raise0.2ex\hbox{$\subset$}\cr\hidewidth\raise-0.8ex\hbox{\scalebox{0.9}{$\sim$}}\hidewidth\cr}}}
\newcommand{\eps}{\varepsilon}

\DeclareMathOperator{\pr}{pr}

\DeclareMathOperator{\End}{End}

\DeclareMathOperator{\Prob}{Prob}

\DeclareMathOperator{\IP}{IP}

\renewcommand{\phi}{\varphi}

\definecolor{lichtgrijs}{gray}{0.95}
\mdfdefinestyle{mystyle}{ %
    backgroundcolor=lichtgrijs, %
    linewidth=0pt, %
    innertopmargin=10pt, %
    innerbottommargin=10pt,%
    nobreak=true% prevent page breaks in the middle of mystyle
}

\theoremstyle{theorem}
\newtheorem{theorem}{Theorem}[section]
\newtheorem{conjecture}{Conjecture}[section]
\newtheorem{corollary}[theorem]{Corollary}
\newtheorem{proposition}[theorem]{Proposition}
\newtheorem{lemma}[theorem]{Lemma}
\theoremstyle{task}

\theoremstyle{definition}

\newtheorem{remark}[theorem]{Remark}

\newtheorem{example}[theorem]{Example}
\newtheorem*{question}{Question}

\usepackage{changepage}
\usepackage{mathtools}
\usepackage{graphicx}
\usepackage{calc}

\usepackage[toc,page]{appendix}

\usepackage{scalerel}

\usepackage{extarrows}

\DeclareMathSymbol{\shortminus}{\mathbin}{AMSa}{"39}
\usepackage{mathtools}
\begin{document}
\bibliographystyle{plain} % Choose Phys. Rev. style for bibliography

\title[Szemer\'edi's Theorem in approximate lattices]{A Szemer\'edi type theorem for sets of positive density in approximate lattices}

%  Author I information
\author{Michael Bj\"orklund}
\address{Department of Mathematics, Chalmers, Gothenburg, Sweden}
\email{micbjo@chalmers.se}
\thanks{}

%    Author II information
\author{Alexander Fish}
\address{School of Mathematics and Statistics F07, University of Sydney, NSW 2006,
Australia}
\curraddr{}
\email{alexander.fish@sydney.edu.au}
\thanks{}


\begin{abstract}
An extension of Szemer\'edi's Theorem is proved for sets of positive density in approximate lattices in general locally compact and second countable abelian groups. As a consequence, we establish a recent conjecture of Klick, Strungaru and Tcaciuc. Via a novel version of Furstenberg's Correspondence principle, which should be of independent interest, we show that our Szemer\'edi Theorems can be deduced from a general \emph{transverse} multiple recurrence theorem, which we establish using recent works of Austin.
\end{abstract}

\maketitle

\section{Introduction}

\subsection{Szemer\'edi's Theorem and the aim of the paper}

The \emph{upper asymptotic density} $\overline{d}(P_o)$ of a subset $P_o \subset \bZ$
is defined by
\[
\overline{d}(P_o) = \varlimsup_{n \ra \infty} \frac{|P_o \cap [-n,n]|}{2n+1}.
\]
The celebrated theorem of Szemer\'edi \cite{Sz}, with the ergodic-theoretical proof of 
Furstenberg \cite[Theorem 11.13]{F}, asserts that for every $P_o \subset \bZ$ with positive upper asymptotic density and for every finite set $F \subset \bZ$, the set 
\[
S_F = \{ \lambda \in \bZ \, : \, \textrm{$\exists \, \lambda_o \in P_o$ such that} \enskip \lambda_o + \lambda \cdot F \subset P_o \}
\]
is syndetic in $\bZ$, i.e. there is a finite set $K \subset \bZ$ such that $S_F + K = \bZ$. In this article we show that a version of Szemer\'edi's Theorem still holds if one replaces the ring $\bZ$ with an \emph{approximate ring}. \\

More precisely, let $R$ be a locally compact ring, e.g. $R = \bR,\bC,\bQ_p$ (for some prime number $p$) or the Hamilton quaternions $\bH$, and suppose $\Lambda \subset R$ is a symmetric, uniformly discrete and multiplicatively closed subset of $R$, which is an approximate subgroup under addition, i.e. it contains $0$ and there is a finite set $K \subset R$ such that $\Lambda + \Lambda \subset \Lambda + K$. Let $\Lambda^\infty \subset R$ denote the additive group generated by $\Lambda$ inside $(R,+)$. The aim of this paper is to provide a partial (affirmative) answer to the following question (see Corollary \ref{cor_ring} below). 

\begin{question}
Let $P_o \subset \Lambda$ be a "large" subset. Given a finite set $F \subset \Lambda^\infty$, define
\[
S_F = \{ \lambda \in \Lambda \, : \, \textrm{$\exists \, \lambda_o \in P_o$ such that} \enskip \lambda_o + \lambda \cdot F \subset P_o \}.
\]
Is $S_F$ a syndetic subset of $\Lambda$ for every $F$? In other words, can we find, for every finite set $F \subset \Lambda^\infty$, a finite subset $K_F \subset R$ such that 
$\Lambda \subseteq S_F + K_F$?
\end{question}

\begin{remark}
Note that we do not require the finite set $F$ to be contained in $\Lambda$, but only in $\Lambda^\infty$, which is typically a \emph{dense} subgroup of $R$.
\end{remark}

The exact definition of "large" will be given below. However, we want to emphasize that 
for the examples we have in mind, the set $\Lambda$ typically has zero upper Banach density inside the group $\Lambda^\infty$, so it is in general very difficult to deduce anything about patterns inside $\Lambda$ from combinatorial arguments within $\Lambda^\infty$. In the next subsection, we provide some explicit classes of examples in $R = \bR$ and $R = \bQ_p$, for which we can answer the question above. We then present the general framework for our investigations, and our main theorems, along with a proof of a recent conjecture by Klick, Strungaru and Tcaciuc \cite{KST}.

\subsection{Some motivating examples}

\subsubsection{Examples from number fields}

To keep things simple, let $D > 0$ be a square-free integer, and define $\Lambda \subset \bR$ by
\[
\Lambda = \{ m + n \sqrt{D} \, : \, |m-n\sqrt{D}| \leq 1 \}.
\]
It is not hard to check that $\Lambda$ is symmetric, uniformly discrete in $\bR$ and multiplicatively closed. Furthermore, 
\[
\Lambda + \Lambda \subset \Lambda + \{-1,0,1\},
\]
so $\Lambda$ is an approximate subgroup of $\bR$ and $\Lambda^\infty = \bZ[\sqrt{D}]$.
We say that a subset $P_o \subset \Lambda$ has \emph{positive upper asymptotic density} 
if 
\[
\varlimsup_{n \ra \infty} \frac{|P_o \cap [-n,n]|}{2n} > 0,
\]
where $[-n,n]$ is viewed as an interval in $\bR$.
The following result is a special case of Corollary \ref{cor_ring} below. 

\begin{theorem}
\label{Thm_R}
Suppose $P_o \subset \Lambda$ has positive upper asymptotic density. Then, for every finite set $F \subset \bZ[\sqrt{D}]$, the set
\[
S_F = \{ \lambda \in \Lambda \, : \, \textrm{$\exists \, \lambda_o \in P_o$ such that} \enskip \lambda_o + \lambda \cdot F \subset P_o \}
\]
is syndetic in $\Lambda$. 
\end{theorem}

\begin{remark}
The non-emptiness of the set $S_F$ can also be proved using the IP-version of Szemer\'edi's Theorem due to Furstenberg and Katznelson \cite{FK2}. We sketch the alternative proof of Theorem \ref{Thm_R} in the appendix. This proof can also be extended to the more general setting of cut-and-project sets discussed below (see the appendix for more details). However, for locally compact and second countable (lcsc) abelian groups like $\bQ_p$ without \emph{dense finitely generated subgroups}, it seems difficult (but certainly not impossible) to establish syndeticity of sets like $S_F$ using this approach, and we believe that such a proof might provide a way to prove Conjecture \ref{Conj} below. Our proof of syndeticity of $S_F$ instead employs some recent results of Austin \cite{austin}.
\end{remark}

\subsubsection{$p$-adic examples}

Let $p$ be a prime number and define the set $\Lambda_p \subset \bQ_p$ by
\[
\Lambda_p = \{ \gamma \in \bZ[1/p] \, : \, |\gamma|_\infty \leq 1 \},
\]
where $|\cdot|_\infty$ denotes the (real) absolute value when $\bZ[1/p]$ is viewed as a sub-ring of $\bR$. Even though $\bQ_p$ does not admit a lattice, the set $\Lambda_p$ is 
nevertheless an approximate lattice, i.e. it is symmetric, uniformly discrete and relatively dense \emph{approximate subgroup} of $\bQ_p$. Furthermore, $\Lambda_p$ is multiplicatively closed and $\Lambda_p^\infty = \bZ[1/p]$. The set $\Lambda_p$ is sometimes referred to as the \emph{fractional parts} of $\bQ_p$. \\

Note that
the sequence $(F_n)$ defined by $F_n = p^{-n} \bZ_p \subset \bQ_p$ is a F\o lner sequence in $\bQ_p$ and $m_{\bQ_p}(F_n) = p^n$, where $m_{\bQ_p}$ denotes the Haar measure on $\bQ_p$, normalized so that $m_{\bQ_p}(\bZ_p) = 1$. We say that a subset 
$P_o \subset \Lambda_p$ has \emph{positive upper asymptotic density} if 
\[
\varlimsup_{n \ra \infty} \frac{|P_o \cap F_n|}{p^n} > 0.
\]
The following result is a special case of Corollary \ref{cor_ring} below. 

\begin{theorem}
\label{Thm_Qp}
Suppose $P_o \subset \Lambda_p$ has positive upper asymptotic density. Then, for every finite set $F \subset \bZ[1/p]$, the set
\[
S_F = \{ \lambda \in \Lambda \, : \, \textrm{$\exists \, \lambda_o \in P_o$ such that} \enskip \lambda_o + \lambda \cdot F \subset P_o \}
\]
is syndetic in $\Lambda_p$. 
\end{theorem}

\subsection{General framework}

Both of the examples above are special cases of the \emph{cut-and-project construction}, which we now review in general. Let $G$ and $H$ be locally compact 
and second countable (lcsc) abelian groups, and let $\Gamma < G \times H$ be a (uniform) lattice such that the projection of $\Gamma$ to $G$ is injective and the projection of $\Gamma$ to $H$ is dense. Let $W \subset H$ be a bounded and symmetric Borel set with non-empty interior, \emph{containing $0$}, and define the \emph{cut-and-project set} $\Lambda = \Lambda(G,H,\Gamma,W)$ by
\[
\Lambda = \pr_G(\Gamma \cap (G \times W)) \subset G,
\]
where $\pr_G$ denotes the projection $(g,h) \mapsto g$ from $G \times H$ to $G$.
We stress that our assumptions that $W$ is symmetric and its interior contains $0$
are not standard, and in most of the literature, cut-and-project sets are defined without these assumptions. They are not strictly necessary for our theorems, but do simplify their exposition. \\

The examples in the previous sub-section are of this form with
\[
G = H = \bR, \enskip \Gamma = \{ (m+n\sqrt{D},m-n\sqrt{D}) \, : \, m,n \in \bZ\}, \enskip W = [-1,1],
\]
and
\[
G = \bQ_p, \enskip H = \bR, \enskip \Gamma = \{ (\gamma,\gamma) \, : \, \gamma \in \bZ[1/p] \}, \enskip W = [-1,1]
\]
respectively. It is not difficult that every $\Lambda$ of this form is uniformly discrete and relatively dense (i.e. there is a compact set $Q \subset G$ such that $Q + \Lambda = G$). Furthermore, there is a finite $F \subset G$ such that
\[
\Lambda - \Lambda \subset \Lambda + F.
\] 
The study of cut-and-project sets was initiated by Yves Meyer \cite{Meyer}. \\

We say that a uniformly discrete and relatively dense subset $\Lambda \subset G$ is a (uniform) \emph{approximate lattice} if it is symmetric, $0 \in \Lambda$ and there is a finite set $F$ in $G$ such that $\Lambda + \Lambda \subset \Lambda + F$. Note that if $\Lambda$ is an approximate lattice, then for every $q \geq 1$, the $q$-iterated sum 
\[
\Lambda^q = \underbrace{\Lambda + \ldots + \Lambda}
\]
is contained in a finite union of translates of $\Lambda$, and is thus again 
uniformly discrete (and relatively dense) in $G$. Furthermore, every cut-and-project set $\Lambda(G,H,\Gamma,W)$, for a \emph{symmetric} set $W \subset H$ whose interior contains $0$, is an approximate lattice. \\

Fix a Haar measure $m_G$ on $G$. Given a strong F\o lner sequence $(F_n)$ (see Subsection \ref{Subsec:Folner} for definitions), we define the \emph{upper asymptotic density $\overline{d}_{(F_n)}(P_o)$} of a uniformly discrete subset $P_o \subset G$ 
by
\[
\overline{d}_{(F_n)}(P_o) = \varlimsup_{n \ra \infty} \frac{|P_o \cap F_n|}{m_G(F_n)},
\]
and the \emph{upper Banach density $d^*(P_o)$} by
\[
d^*(P_o) = \sup\big\{ \overline{d}_{(F_n)}(P_o) \, : \, \textrm{$(F_n)$ is a strong F\o lner sequence in $G$}\}.
\]
We show below (Subsection \ref{subsec:Banach}) that $d^*(P_o)$ is always finite. 

\subsection{Main combinatorial results}

We can now state our first main result, whose proof will be given in Section \ref{sec:prfcomb}. Let $G$ be a lcsc abelian group, and 
denote by $\End(G)$ the space of continuous endomorphisms on $G$.

\begin{theorem}
\label{Thm_MainComb}
Let $\Lambda \subset G$ be an approximate lattice and suppose there is a set $\Lambda_o \subset \Lambda$ with positive upper Banach density such that $\Lambda_o - \Lambda_o \subset \Lambda$. 
Let $P_o \subset \Lambda$ be a subset with positive Banach upper density. Suppose that there exist $\alpha_1,\ldots,\alpha_r \in \End(G)$ and $q \geq 1$ such that $\alpha_k(\Lambda) \subset \Lambda^q$ for all $k$. Then there exists $c > 0$ such that the set
\[
S = \Big\{ \lambda \in \Lambda \, : \, d^*\Big(P_o \cap \Big( \bigcap_{k=1}^r (P_o - \alpha_k(\lambda)) \Big)\Big) \geq c \Big\}
\]
is syndetic in $\Lambda$. 
\end{theorem}

\begin{remark}
\label{Rmk_NoPwrs}
If $\Lambda$ is an approximate lattice in $G$ which is also a cut-and-project set of the form $\Lambda(G,H,\Gamma,W)$ for some \emph{symmetric} set $0 \in W^o \subset W \subset H$ then we can 
choose $\Lambda_o = \Lambda(G,H,\Gamma,W_o)$ in Theorem \ref{Thm_MainComb}, where $W_o$ is an identity neighbourhood in $H$ such that $W_o - W_o \subset W^o$. Note that in this case, 
$\Lambda_o$ is relatively dense. However, there are examples of approximate lattices (already for $G = \bR$) which do not contain the difference set of any relatively dense set
\cite[Example, Subsection 3.1]{BH2}. These examples do however contain the difference set of a set of positive upper Banach density, so the theorem above is applicable. We
further stress that these examples do \emph{not} contain any higher order difference sets, i.e. they do not contain sets of the form $\Lambda_o - \Lambda_o + \Lambda_o - \Lambda_o$, for some subset $\Lambda_o$ of positive upper Banach density.
\end{remark}

If $G = (R,+)$ for some locally compact ring $R$, we can deduce the following corollary 
from Theorem \ref{Thm_MainComb}.

\begin{corollary}
\label{cor_ring}
Let $R$ be a locally compact and second countable ring, and let $\Lambda \subset R$
be a multiplicatively closed cut-and-project set. Suppose $P_o \subset \Lambda$
has positive upper Banach density in $(R,+)$. Then, for every finite set $F \subset \Lambda^\infty$, the set
\[
S_o = \{ \lambda \in \Lambda \, : \lambda_o + \lambda \cdot F \subset P_o \enskip \textrm{for some $\lambda_o \in P_o$} \}
\]
is syndetic in $\Lambda$.
\end{corollary}

\begin{proof}
Let $\Lambda = \Lambda(R,H,\Gamma,W) \subset R$ for some $H,\Gamma$ and $W$, and suppose $\Lambda$ is multiplicatively closed, i.e. $\Lambda \cdot \Lambda \subset \Lambda$.  Since we assume that the set $W$ has non-empty interior and contains $0$, we can find 
a symmetric and open set $W_o \subset W$ which contains $0$ such that $W_o - W_o \subset W$. If we define $\Lambda_o = \Lambda_o(R,H,\Gamma,W_o)$, then $\Lambda_o \subset \Lambda$ is an approximate lattice in $R$ (and thus has positive upper Banach density) and
\[
\Lambda_o - \Lambda_o \subset \Lambda.
\]
Let $F = \{\lambda_1,\ldots,\lambda_r\} \subset \Lambda^\infty$ be a finite set, and let 
$q \geq 1$ denote the smallest integer such that $F \subset \Lambda^q$. 
Define $\alpha_1\ldots,\alpha_r \in \End(R)$ by $\alpha_k(g) = g \lambda_k$ for $k=1,\ldots,r$. Then, since $\Lambda \cdot \Lambda \subset \Lambda$, we have
\[
\alpha_k(\Lambda) \subset \Lambda \cdot \Lambda^q \subset \Lambda^q, \quad \textrm{for all $k=1,\ldots,r$}.
\]
Hence all conditions of Theorem \ref{Thm_MainComb} are satisfied, and there is a constant $c > 0$ such that
\[
S = \Big\{ \lambda \in \Lambda \, : \, d^*\Big( P_o \cap \Big( \bigcap_{k=1}^r \big(P_o - \lambda \lambda_k \big) \Big) \Big) \geq c \}
\]
is syndetic in $\Lambda$. Clearly, $S \subset S_o$, so we are done.
\end{proof}

\subsection{Solution to a recent conjecture of Klick, Strungaru and Tcaciuc}

Theorem \ref{Thm_MainComb} also provides a solution to the following conjecture, recently made by Klick, Strungaru and Tcaciuc \cite[Conjecture 6.7]{KST}. Recall
that a subset $\{\lambda_1,\ldots,\lambda_r\}$ in a lcsc group $G$ is an 
\emph{arithmetic progression of length $r$} if $\lambda = \lambda_{k+1} - \lambda_k \neq 0$ for all $k=1,\ldots,r-1$, i.e. if $\lambda_k = \lambda_o + k \cdot \lambda$
for all $k=1,\ldots,r$, for some $\lambda_o$ and $\lambda \in G$. We refer to 
$\lambda$ as the \emph{gap} of the arithmetic progression.

\begin{conjecture}
Let $G$ be a lcsc group and let $\Lambda \subset G$ be a cut-and-project set.
Then every subset $P_o \subset \Lambda$ with positive upper Banach density 
contains arbitrary long arithmetic progressions.
\end{conjecture}

Let us now briefly show how Theorem \ref{Thm_MainComb} establishes the following strengthening of this conjecture. 

\begin{theorem}
Let $G$ be a lcsc group, let $\Lambda \subset G$ be a cut-and-project set and 
suppose that $P_o \subset \Lambda$ has positive upper Banach density. Then, for
every $r \geq 1$, there exist a syndetic set $S_r \subset \Lambda$ with the property
that for every $\lambda \in S_r$, there is a set $P_\lambda \subset P_o$ with positive upper Banach density such that
\[
\lambda_o + \{\lambda,2 \lambda,\ldots,r\lambda\} \subset P_o, \quad \textrm{for all $\lambda_o \in P_\lambda$}.
\]
In particular, $P_o$ contains
arbitrary long arithmetic progressions, and the set of all possible gaps of these 
arithmetic progressions contains a syndetic subset of $\Lambda$.
\end{theorem}

\begin{proof}
Let $\Lambda \subset G$ be a cut-and-project set and let $P_o \subset \Lambda$ be a
set of positive upper Banach density. We may without loss of generality assume that $0 \in \Lambda$. Fix $r \geq 1$, and define $\alpha_1,\ldots,\alpha_r \in \End(G)$ by
\[
\alpha_k(g) = k \cdot g, \quad \textrm{for $k=1,\ldots,r$}.
\]
Then, since $\Lambda  \subset \Lambda^2 \subset \ldots \subset \Lambda^r$, we have $\alpha_k(\Lambda) \subset \Lambda^{r}$ for all $k=1,\ldots,r$, so the conditions of 
Theorem \ref{Thm_MainComb} are satisfied with $q=r$, so we conclude that there exists a syndetic set $S_r \subset \Lambda$ with the property that for every $\lambda \in S_r$, the set
\[
P_\lambda = P_o \cap \Big( \bigcap_{k=1}^r (P_o - k \lambda) \Big) \subset P_o
\]
has positive upper Banach density. 
\end{proof}

\subsection{Main dynamical result}

In Section \ref{subsec:corr} we establish a version of Furstenberg's famous Correspondence Principle for uniformly discrete subsets of a lcsc abelian group $G$.
Using this Correspondence Principle, it is not hard to deduce Theorem \ref{Thm_MainComb} from the following dynamical statement, whose proof will be given in Section \ref{sec:prfdyn}. \\

Let $(X,\mathscr{B}_X)$ be a standard Borel space, equipped with a Borel measurable $G$-action $G \times X \ra X, \enskip (g,x) \mapsto g.x$. We say that a Borel set $Y \subset X$ is a \emph{separated cross-section} if $G.Y = X$ and $\Xi_Y \cap U = \{0\}$, for some open identity neighbourhood $U$ in $G$, where $\Xi_Y$ denotes the set of return times, defined by
\[
\Xi_Y = \{ g \in G \, : \, g.Y \cap Y \neq \emptyset \} \subset G.
\]
For every $\mu \in \Prob_G(X)$, there exists a unique non-negative Borel measure 
$\mu_Y$ on $Y$ (called the \emph{transverse measure} associated to $\mu$) such that 
$\mu(V.B) = m_G(V) \cdot \mu_Y(B)$ for every Borel set $V \subset G$ such that $V-V \subset U$ and for every Borel set $B \subset Y$. It is well-known that a separated cross-section always exists. We refer to Subsection \ref{subsec:cross} for more details.

\begin{theorem}[Multiple recurrence for cross-sections]
\label{Thm_MainErg}
Let $\Lambda \subset G$ be an approximate lattice and suppose that there exists a set $\Lambda_o \subset \Lambda$ with positive upper Banach density such that $\Lambda_o - \Lambda_o \subset \Lambda$. Let $\alpha_1,\ldots,\alpha_r \in \End(G)$ and suppose that there is
$q \geq 1$ such that $\alpha_k(\Lambda) \subset \Lambda^q$ for all $k$. If there 
is an open identity neighbourhood $V \subset U$
such that $(\Xi_Y - \Lambda^q) \cap V = \{0\}$, then for every
$\mu \in \Prob_G(X)$ and for every Borel set $B \subset Y$
with positive $\mu_Y$-measure, there exists $c > 0$ such that the set
\[
S = \Big\{ \lambda \in \Lambda \, : \, \mu_Y\Big( \bigcap_{k=1}^r (-\alpha_k(\lambda)).B\Big) \geq c \Big\}
\]
is syndetic in $\Lambda$.
\end{theorem}

\subsection{Potential polynomial extensions}

A series of impressive extensions of Szemer\'edi's Theorem have been proved over the years. Polynomial versions have been particularly popular, starting with the ground-breaking works of Bergelson, Leibman and McCutcheon \cite{BL1, BL2,BM}. In view of these works, specifically \cite[Theorem 0.3]{BM}, we feel that it is natural to formulate the following conjectural polynomial strengthening of Corollary \ref{cor_ring}. If $R$ is a locally compact ring, a 
\emph{polynomial} is a map $p : R \ra R$ of the form
\[
p(x) = \sum_{k=0}^r a_k x^k, \quad \textrm{for some $a_o, a_1,\ldots,a_r \in R$}.
\]

\begin{conjecture}
\label{Conj}
Let $R$ be a locally compact and second countable ring and let $\Lambda \subset R$
be a multiplicatively closed cut-and-project set. Suppose that $P_o \subset \Lambda$ has
positive upper Banach density in $(R,+)$. Then, for every $r \geq 1$, for all polynomials $p_1,\ldots,p_r : R \ra R$ such that $p_k(0) = 0$ and 
$p_k(\Lambda) \subset \Lambda^q$ for some $q \geq 1$ and for all $k$, and 
for every finite set $F \subset \Lambda^\infty$, the set
\[
S_F = \{ \lambda \in \Lambda \, : \, \textrm{$\exists \, \lambda_o \in P_o$ such that}\enskip \lambda_o + \{p_1(\lambda),\ldots,p_r(\lambda)\} \subset P_o \}
\]
is syndetic.
\end{conjecture}

The arguments in this paper, more specifically the ones to prove Proposition \ref{Prop_LowBndAvLambda} below, are currently not flexible enough to prove this conjecture. \\

Similarly, given that much of the ergodic-theoretical machinery can be extended to actions of nilpotent groups, it is not be unreasonable to expect that some version of Theorem \ref{Thm_MainComb} (and Theorem \ref{Thm_MainErg}) could be extended to cut-and-project sets in lcsc \emph{nilpotent} groups. For more details about such sets, we refer the reader to \cite{BH1,BH2}.

\subsection{Organization of the paper}

In Section \ref{sec:Prel} we introduce notation and some fundamental concepts like cross-sections, transverse measures, Banach densities and strong F\o lner sequences. 
The most important result in Section \ref{subsec:corr} is a transverse version of Furstenberg's Correspondence Principle (Theorem \ref{Thm_CorrPrinciple}). In Section \ref{Sec:FK} we extend Furstenberg and Katznelson's multiple recurrence theorem to general locally compact and second countable abelian groups, using recent works of Austin. In Section \ref{sec:prfdyn}, we prove our main dynamical theorem (Theorem \ref{Thm_MainErg}), 
and in Section \ref{sec:prfcomb} we prove our main combinatorial theorem (Theorem \ref{Thm_MainComb}). Finally, in the appendix we show how a weaker version of Theorem \ref{Thm_MainComb} for cut-and-project sets in $\bR^d$ can be deduced from Furstenberg and Katznelson's IP-theory developed in the seminal paper \cite{FK2}. 

\subsection{Acknowledgements}

This work was completed during the M.B's research visit to SMRI (Sydney Mathematical Research Institute), and he wishes to express his gratitude to this organization for its hospitality. M.B is grateful to John Griesmer for useful discussions on the topic of the paper. M.B. was supported by the grants 11253320 and 2023-03803 from the Swedish Research Council. A.F. was supported by the ARC via grants DP210100162 and DP240100472.
Finally, we are grateful to the referee who provided a very careful report. 

\section{Preliminaries}
\label{sec:Prel}

Let $G$ be a locally compact and second countable (lcsc) abelian group. In particular, as a topological space, $G$ is $\sigma$-compact \cite[Theorem 2.B.4]{CornHarpe}. Let us also fix a Haar measure $m_G$ on $G$ for the rest of our discussions. 

\subsection{Borel $G$-spaces, separated cross-sections and transverse measures}
\label{subsec:cross}

Let $(X,\mathscr{B}_X)$ be a standard Borel space, equipped with a Borel measurable action $a : G \times X \ra X$. We refer to $(X,\mathscr{B}_X)$ as a \emph{Borel $G$-space}. To make the notation less heavy, we often write $a(g,x) = g.x$. We denote by $\Prob_G(X)$ the space of $G$-invariant Borel probability measures on $X$. A Borel set $Y \subset X$ is a \emph{cross-section} (or \emph{transversal}) if $G.Y = X$, and we denote by $\Xi_Y \subset G$ the set of \emph{return times}, defined by
\[
\Xi_Y = \{ g \in G \, : \, g.Y \cap Y \neq \emptyset \}.
\]
If $U$ is an open identity neighbourhood in $G$, we say that a cross-section 
$Y$ is \emph{$U$-separated} if $U \cap \Xi_Y = \{0\}$. It is well-known (see e.g. \cite[Theorem 2.4]{Slutsky}) that every Borel $G$-space admits $U$-separated cross-sections for some open identity neighbourhood. \\

In what follows, we fix an open and bounded identity neighbourhood $U$ in $G$, and 
a $U$-separated cross-section $Y \subset X$. The following result is standard (see e.g. \cite[Proposition 4.3]{BHK}).

\begin{lemma}
\label{Lemma_transmeas}
For all Borel sets $V \subset G$ and $B \subset Y$, the action set $V.B \subset X$
is Borel measurable. Furthermore, for every $\mu \in \Prob_G(X)$, there exists a \emph{unique} non-negative finite Borel measure $\mu_Y$ on $Y$ such that
\[
\mu(V.B) = m_G(V) \cdot \mu_Y(B), 
\]
for all Borel sets $V \subset G$ and $B \subset Y$ such that $V-V \subset U$.
\end{lemma}

\begin{remark}
Note that we can always view $\mu_Y$ as a Borel measure on $X$ by defining $\mu_Y(B) = \mu_Y(B \cap Y)$
for $B \in \mathscr{B}_X$.
\end{remark}

Given $\mu \in \Prob_G(X)$, we refer to the Borel measure $\mu_Y$ (which is finite, but not a probability measure in general) as the \emph{transverse measure associated to $\mu$}. Note that for any Borel set $B \subset Y$, we have
\[
\mu_Y(B) > 0 \iff \mu(V.B) > 0, 
\]
for some (any) identity neighbourhood $V$ in $G$. The following lemma will be used in several places in the proofs below. 

\begin{lemma}
\label{Lemma_g1gr}
Let $\Delta \subset G$ and suppose that there is an open identity neighbourhood $W \subset U$ such that $(\Xi_Y - \Delta) \cap W = \{0\}$. For every Borel set $V_o$ in $G$ 
such that $V_o-V_o \subset W$, and for every Borel set $B \subset Y$, 
\[
V_o.\Big(B \cap \bigcap_{k=1}^r (-g_k).B\Big) = B_{V_o} \cap \Big( \bigcap_{k=1}^r (-g_k).B_{V_o}\Big), \quad \textrm{for all $g_1,\ldots,g_r \in \Delta$},
\]
where $B_{V_o} = V_o.B$. In particular, for every $\mu \in \Prob_G(X)$, 
\[
\mu_Y\Big( B \cap \bigcap_{k=1}^r (-g_k).B\Big)
= \frac{1}{m_G(V_o)} \cdot \mu\Big(B_{V_o} \cap \Big( \bigcap_{k=1}^r (-g_k).B_{V_o}\Big)\Big).
\]
\end{lemma}

\begin{proof}
Fix $V_o$ and $B$ as in the lemma and an $r$-tuple $g_1,\ldots,g_r \in \Delta$. We first note that the inclusion 
\[
V_o.\Big(B \cap \bigcap_{k=1}^r (-g_k).B\Big) \subseteq B_{V_o} \cap \Big( \bigcap_{k=1}^r (-g_k).B_{V_o}\Big)
\] 
is trivial. In particular, if the right-hand side is empty, so is the left-hand side. It thus remains to show that if the right-hand side is non-empty, then
\begin{equation}
\label{rightincl}
V_o.\Big(B \cap \bigcap_{k=1}^r (-g_k).B\Big) \supseteq B_{V_o} \cap \Big( \bigcap_{k=1}^r (-g_k).B_{V_o}\Big).
\end{equation}
To do this, fix an element $x$ in the set on the right-hand side of \eqref{rightincl}. By definition, we can find $v_o,v_1,\ldots,v_r \in V_o$ and $y_o,y_1,\ldots,y_r \in B$ such that
\[
x = v_o.y_o = (v_1-g_1).y_1 = \ldots = (v_r - g_r).y_r.
\]
In particular, $v_o - v_k + g_k \in \Xi_Y$ for all $k=1,\ldots,r$ and thus
\[
v_o - v_k \in (\Xi_Y - g_k) \cap (V_o - V_o) \subseteq (\Xi_Y - \Delta) \cap W = \{0\}, \quad \textrm{for all $k=1,\ldots, r$}.
\]
We conclude that $v_o = v_1 = \ldots = v_r$, and thus $y_o = (-g_1).y_1 = \ldots = (-g_r).y_r$. In particular, 
\[
x = v_o.y_o \in v_o.\Big(B \cap \bigcap_{k=1}^r (-g_k).B\Big) \subset V_o.\Big(B \cap \bigcap_{k=1}^r (-g_k).B\Big).
\]
Since $x$ is arbitrary, this shows \eqref{rightincl}, and we are done.
\end{proof}

\subsection{The Chabauty topology and hulls of uniformly discrete sets}
\label{subsec:cha}

We denote by $\mathscr{C}(G)$ the space of closed subsets of $G$, endowed with the Chabauty topology. We recall from \cite[Section E.1]{BP} that 
$\mathscr{C}(G)$ is a compact metrizable space, and a sequence $(P_n)$ in $\mathscr{C}(G)$ converges to some $P \in \mathscr{C}(G)$ if and only if the following two conditions hold:
\begin{itemize}
\item[(i)] If $p_{n_k} \in P_{n_k}$ for some sub-sequence $(n_k)$ and $p_{n_k} \ra p$ in $G$, then $p\in P$ \vspace{0.1cm}
\item[(ii)] For every $p \in P$, there exist $p_n \in P_n$ such that $p_n \ra p$ in $G$.
\end{itemize}
We equip $\mathscr{C}(G)$ with the jointly continuous $G$-action:
\[
(g,P) \mapsto g.P = P - g, \quad (g,P) \in G \times \mathscr{C}(G).
\]
Given $P_o \in \mathscr{C}(G)$, we define
\[
\Omega_{P_o} := \overline{G.P_o} \qand \Omega_{P_o}^{\times} := \Omega_{P_o} 
\setminus \{\emptyset\},
\]
and refer to $\Omega_{P_o}$ as the \emph{hull of $P_o$} and to $\Omega_{P_o}^{\times}$
as the \emph{punctured hull of $P_o$}. Note that the space $\Omega_{P_o}$ is compact, while $\Omega_{P_o}^{\times}$ is only compact if $\emptyset \notin \Omega_{P_o}$. We also define the \emph{canonical cross-section $\cT_{P_o} \subset \Omega_{P_o}^{\times}$}
by
\[
\cT_{P_o} = \{ P \in \Omega_{P_o}^{\times} \, : \, 0 \in P \}.
\]
Note that $\cT_{P_o}$ is a compact subset of $\Omega_{P_o}^{\times}$, and for every open set $V \subset G$, the set
\[
\mathcal{O}_V := (-V).\cT_{P_o} = \{ P \in \Omega_{P_o} \, : \, P \cap V \neq \emptyset \}
\]
is open in both $\Omega_{P_o}$ and $\Omega_{P_o}^{\times}$. \\

We say that $P \subset G$ is \emph{uniformly discrete} if $(P - P) \cap U = \{0\}$ for some identity neighbourhood $U$ in $G$. If we want to emphasize the dependence on $U$, we say that the set $P$ is \emph{$U$-uniformly discrete}. We say that $P \subset G$
\emph{has finite local complexity} if $P-P$ is locally finite. Clearly every set with finite local complexity is uniformly discrete. \\

In what follows, let $P_o \subset G$ be a $U$-uniformly discrete set for some 
open and bounded identity neighbourhood $U$ in $G$.

\begin{lemma}
\label{Lemma_TransSep}
Let $P \in \Omega_{P_o}$. Then,  
\[
P-P \subset \overline{P_o - P_o} \quad \textrm{for all $P \in \Omega^{\times}_{P_o}$},
\] 
and $\cT_{P_o}$ is a $U$-separated cross-section in $\Omega_{P_o}^{\times}$ with 
$\Xi_{\cT_{P_o}} \subset \overline{P_o - P_o}$.
\end{lemma}

\begin{proof}
Fix an open identity neighbourhood $W$ in $G$ and a sequence $(g_n)$ in $G$ such that $g_n.P_o \ra P$ as $n \ra \infty$. Then, by using (i) in the characterization of sequential convergence in the Chabauty topology above, we see that for every pair $p_1,p_2 \in P$,
we can find $w_{1,n}, w_{2,n} \in W$ and $p^{o}_{1,n}, p^{o}_{2,n} \in P_o$ such that
\[
p_1 = p^{o}_{1,n} - g_n + w_{1,n} \qand p_2 = p^{o}_{2,n} - g_n + w_{2,n}, 
\] 
for all sufficiently large $n$. In particular, 
\[
p_1 - p_2 = (p_{1,n}^o + w_{1,n}) - (p_{2,n}^o + w_{2,n}) \subset P_o - P_o  + W-W.
\] 
Since $p_1$ and $p_2$ are arbitrary in $P$, we have $P-P \subset P_o - P_o + W-W$. Furthermore, $W$ is arbitrary, and thus $P-P \subset \overline{P_o - P_o}$. For the last assertion, pick $g \in \Xi_{\cT_{P_o}}$ and an element $P \in \cT_{P_o} \cap g.\cT_{P_o}$. Then $0 \in P$ and $-g \in P$, and thus $g \in P-P$. From the previous inclusion, we conclude that $\Xi_{\cT_{P_o}} \subset \overline{P_o - P_o}$. Since $P_o$
is $U$-uniformly discrete, it follows that $\Xi_{\cT_{P_o}} \cap U = \{0\}$ and thus $\cT_{P_o}$ is $U$-separated.
\end{proof}

If $M$ is a locally compact metrizable space and $\nu$ is a finite non-negative 
Borel measure on $M$, then a Borel set $B \subset M$ is called \emph{$\nu$-Jordan measurable} if $\nu(B^o) = \nu(\overline{B})$. By choosing a continuous metric $\rho$
on $M$ which induces the topology, it is not hard to show that for every $m \in M$,
there are at most countably many radii $r > 0$ such that $\rho$-ball of radius $r$ around $m$ is \emph{not} $\nu$-Jordan measurable. In particular, for every $m \in M$
and open neighbourhood $U$ around $m$, there is a $\nu$-Jordan measurable open 
neighbourhood of $m$ contained in $U$. 

\begin{lemma}
\label{Lemma_Jordan}
Let $P_o \subset G$ be a $U$-uniformly discrete set for some open identity neighbourhood $U$ in $G$ and let $V$ be a non-empty, open and $m_G$-Jordan measurable subset of $G$ such that $\overline{V} - \overline{V} \subset U$. Then, for every $\mu \in \Prob_G(\Omega_{P_o})$, the open set $\cO_V \subset \Omega_{P_o}$ is $\mu$-Jordan measurable.
\end{lemma}

\begin{proof}
Fix $\mu \in \Prob_G(\Omega_{P_o})$ and write $\mu = \mu' + \alpha \cdot 
\delta_{\{\emptyset\}}$ for some finite and non-negative Borel measure $\mu'$ on $\Omega_{P_o}^{\times}$ and $\alpha \geq 0$. Recall that Lemma \ref{Lemma_TransSep} tells us that $\cT_{P_o}$ is a $U$-separated cross-section in $\Omega_{P_o}^{\times}$, and thus the transverse measure $(\mu')_{\cT_{P_o}}$ is well-defined. Since $\cO_V$ is open, $\overline{\cO_V} \subset \cO_{\overline{V}}$ and $\delta_{\{\emptyset\}}(\cO_{\overline{V}}) = 0$, we have
\begin{align*}
\mu(\cO_V) &= \mu'(\cO_V) 
= m_G(V) \cdot (\mu')_{\cT_{P_o}}(\cT_{P_o}) = 
m_G(\overline{V}) \cdot (\mu')_{\cT_{P_o}}(\cT_{P_o}) \\[0.2cm]
&= \mu'(\cO_{\overline{V}}) = \mu(\cO_{\overline{V}}) \geq \mu(\overline{\cO_V}),
\end{align*}
and thus $\cO_V$ is $\mu$-Jordan measurable. Since $\mu$ is arbitrary, we are done.
\end{proof}

We also record the following lemma for future use. 

\begin{lemma}
\label{Lemma_lwbndnu}
Let $\Lambda \subset G$ be a countable set and suppose that there exists a uniformly discrete set $\Lambda_o$ such that $\Lambda_o - \Lambda_o \subset \Lambda$. Then for all identity neighbourhoods $V_o$ and $V$ 
in $G$ such that $V_o - V_o \subset V$ and for all $\nu \in \Prob_G(\Omega^{\times}_{\Lambda_o})$, we have
\[
\sum_{\lambda \in \Lambda} \chi_V(g+\lambda) \geq \nu(\cO_{V_o} \cap g.\cO_{V_o}), \quad \textrm{for all $g \in G$},
\]
where $\chi_V$ denotes the indicator function of $V$, and $\cO_{V_o} = (-V_o).\cT_{\Lambda_o} \subset \Omega_{\Lambda_o}^{\times}$.
\end{lemma}

\begin{proof}
It is clearly enough to show that 
\[
\cO_{V_o} \cap g.\cO_{V_o} \neq \emptyset \implies (g + \Lambda) \cap V \neq \emptyset. 
\]
To do this, note that if the open set $\cO_{V_o} \cap g.\cO_{V_o}$ is non-empty, then, since $\Lambda_o$ has a dense $G$-orbit in $\Omega_{\Lambda_o}^{\times}$, there is $g_o \in G$ such that
$g_o.\Lambda_o \in \cO_{V_o} \cap g.\cO_{V_o}$, or equivalently,
\[
(\Lambda_o - g_o) \cap V_o \neq \emptyset \qand (\Lambda_o - g_o+g) \cap V_o \neq \emptyset.
\]
By taking difference, this implies that 
\[
\emptyset \neq (g + \Lambda_o - \Lambda_o) \cap (V_o - V_o) \subset (g + \Lambda) \cap V,
\]
and we are done.
\end{proof}

\subsection{Strong F\o lner sequences}
\label{Subsec:Folner}
Let $G$ be a lcsc abelian group and let $(F_n)$ be a sequence of bounded Borel sets with positive Haar measures. We say that $(F_n)$ is \emph{a strong F\o lner sequence} if
\vspace{0.1cm}
\begin{itemize}
\item[(i)] for every compact subset $C \subset G$, we have $\lim_{n \ra \infty} \frac{m_G(F_n+C)}{m_G(F_n)} = 1$, \vspace{0.2cm}
\item[(ii)] there exists an open bounded identity neighbourhood $V \subset G$ such that
\[
\lim_{n \ra \infty} \frac{m_G\big( \bigcap_{v \in V} (F_n-v) \big)}{m_G(F_n)} = 1.
\]
\end{itemize} 
If we want to emphasize the dependence on the identity neighbourhood $V$, we refer to 
the sequence $(F_n)$ as a \emph{$V$-adapted strong F\o lner sequence}.

\begin{remark}
Note that if $G$ is a countable discrete group, then every F\o lner sequence in $G$
is a strong F\o lner sequence (we can take $V = \{e\}$).
\end{remark}

\begin{lemma}
\label{Lemma_Folner}
Let $V$ be an open and bounded identity neighbourhood in $G$. Then there 
exists a $V$-adapted strong F\o lner sequence in $G$. Furthermore, if $(L_n)$ is a 
$V$-adapted strong F\o lner sequence in $G$ and $(g_n)$
is any sequence in $G$, then 
\[
F_n := L_n + g_n \qand \widetilde{F}_n := L_n + V 
\]
are both $V$-adapted strong F\o lner sequences.
\end{lemma}

\begin{proof}
Since $G$ is both amenable and $\sigma$-compact, 
\cite[Theorem 4 and Proposition 1]{Emerson} guarantee that there exists a 
sequence $(K_n)$ of bounded Borel sets in $G$ with positive Haar measures 
such that
for every compact subset $B \subset G$, 
\begin{equation}
\label{emcond}
\lim_{n \ra \infty} \frac{m_G((K_n+B) \Delta K_n)}{m_G(K_n)} = 0,
\end{equation}
where $\Delta$ denotes the set-theoretical difference. In particular, if $0 \in B$,
then
\begin{equation}
\label{limit1}
1 \leq \frac{m_G(K_n + B)}{m_G(K_n)} = 1 + \frac{m_G((K_n + B) \setminus K_n)}{m_G(K_n)} \ra 1, \quad \textrm{as $n \ra \infty$}.
\end{equation}
Let $V$ be an open bounded identity neighbourhood in $G$ and define $F_n = K_n + V$. We claim 
that $(F_n)$ is a $V$-adapted strong F\o lner sequence in $G$. Let us first show
that $(F_n)$ satisfies $(i)$. Let $C$ be a compact subset of $G$. Since 
$V$ is pre-compact, $B := \overline{C + V} \cup \{0\}$ is compact and $F_n + C \subset K_n + B$, and thus \eqref{emcond} implies that
\[
1 \leq \frac{m_G(F_n + C)}{m_G(F_n)} \leq \frac{m_G(K_n + B)}{m_G(K_n)} \ra 1, \quad n \ra \infty,
\]
proving (i). Let us now show (ii). Since $F_n = K_n + V$ and $V$ is a pre-compact identity neighbourhood, \eqref{limit1} tells us that
\[
1 \geq \lim_{n \ra \infty} 
\frac{m_G\big( \bigcap_{v \in V} (F_n-v) \big)}{m_G(F_n)} \geq \lim_{n \ra \infty} \frac{m_G(K_n)}{m_G(K_n + V)} = 1,
\]
and thus (ii) holds for $(F_n)$. For the last assertion, let $(L_n)$ be a $V$-adapted strong F\o lner sequence in $G$, let $(g_n)$ be any sequence in $G$ and let $F_n = L_n + g_n$. Then, if $C$ is a compact set in $G$, we have $F_n + C = (L_n + C) + g_n$, and 
\[
\frac{m_G(F_n + C)}{m_G(F_n)} = \frac{m_G(L_n + C)}{m_G(L_n)} \ra 1, \quad n \ra \infty,
\]
since $(L_n)$ is a strong F\o lner sequence in $G$.
Furthermore, 
\[
\bigcap_{v \in V} (F_n - v) = g_n + \bigcap_{v \in V} (L_n - v), \quad \textrm{for all $n$},
\]
and thus
\[
\lim_{n \ra \infty} \frac{m_G\big( \bigcap_{v \in V} (F_n-v) \big)}{m_G(F_n)}
=
\lim_{n \ra \infty} \frac{m_G\big( \bigcap_{v \in V} (L_n-v) \big)}{m_G(L_n)} = 1,
\]
since $(L_n)$ is $V$-adapted. We conclude that $(F_n)$ is a $V$-adapted strong F\o lner sequence in $G$. The proof that $(\widetilde{F}_n)$ is a $V$-adapted strong F\o lner sequence is similar.
\end{proof}

\subsection{Uniformly discrete sets and Banach densities}
\label{subsec:Banach}
Let $U$ be an open identity neighbourhood in $G$ and let $P \subset G$ be a $U$-uniformly discrete set. Given a strong F\o lner sequence $(F_n)$ in $G$, which always exists in view of Lemma \ref{Lemma_Folner}, we define
the \emph{upper density $\overline{d}_{(F_n)}(P)$ of $P$ along $(F_n)$} by 
\[
\overline{d}_{(F_n)}(P) = \varlimsup_{n \ra \infty} \frac{|P \cap F_n|}{m_G(F_n)},
\]
and the \emph{upper Banach density $d^*(P)$} by
\[
d^*(P) = \sup \big\{ \, \overline{d}_{(F_n)}(P) \, : \, \textrm{$(F_n)$ is a strong F\o lner sequence in $G$} \big\}.
\]
We will show below (Corollary \ref{cor_dstarfinite}) that $d^*(P)$ is always finite. To do this, we first need to prove the following lemma.
\begin{lemma}
\label{Lemma_UQV}
Let $U$ be an open identity neighbourhood and let $P \subset G$ be a $U$-uniformly discrete set. Then, for every  Borel set $V \subset G$ such that $V-V \subset U$
and for every bounded Borel set $Q \subset G$,
\[
|P \cap Q| \leq \frac{m_G(Q+V)}{m_G(V)}.
\]
Furthermore, for every Borel set $V_o \subset V$ such that $V_o - V_o \subset V$, 
\[
|P \cap (Q+V)| \geq \frac{m_G((P+V_o) \cap (Q+V_o))}{m_G(V_o)}.
\]
\end{lemma}

\begin{proof}
Let $V$ be a Borel set which contains $0$ and such that $V - V \subset U$, and let $Q \subset G$ be a bounded Borel set. Then, since $P$ is $U$-uniformly discrete, all sets of the form $p+V$, for $p \in P$ are disjoint, and thus
\begin{equation}
\label{PV}
(P + V) \cap (Q+V) = \bigsqcup_{p \in P} (p + V) \cap (Q+V) \supseteq \bigsqcup_{p \in P \cap Q} (p+V).
\end{equation}
Hence, 
\[
m_G(Q + V) \geq m_G((P + V) \cap (Q+V)) \geq m_G(V) |P \cap Q|.
\]
For the last inequality in the lemma, let $V_o \subset V$ be a Borel set such that $V_o - V_o \subset V$. Then,
\[
(P + V_o) \cap (Q+V_o) 
= 
\bigsqcup_{p \in P} (p + V_o) \cap (Q+V_o) \subseteq \bigsqcup_{p \in P \cap (Q+V)} (p+V_o),
\]
and thus
\[
m_G((P + V_o) \cap (Q+V_o)) \leq m_G(V_o) \cdot |P \cap (Q+V)|.
\]
\end{proof}

\begin{corollary}
\label{cor_dstarfinite}
Let $U$ be an open identity neighbourhood and let $P \subset G$ be a $U$-uniformly discrete set. Then, $d^*(P) \in [0,1/m_G(V)]$ for every identity neighbourhood $V$
in $G$ such that $V-V \subset U$.
\end{corollary}

\begin{proof}
Let $(F_n)$ be a strong F\o lner sequence, and let $V$
be an identity neighbourhood in $G$ such that $V-V \subset U$. By Lemma \ref{Lemma_UQV}, 
\[
\varlimsup_{n \ra \infty}
\frac{|P \cap F_n|}{m_G(F_n)} \leq \varlimsup_{n \ra \infty}
\frac{m_G(F_n+V)}{m_G(V) \cdot m_G(F_n)} = \frac{1}{m_G(V)}.
\]
Since $(F_n)$ is arbitrary, we are done.
\end{proof}

\subsection{Syndetic sets}

Let $\Lambda$ be a (not necessarily closed) subset of $G$. We say that $\Lambda$ is
\emph{relatively dense} in $G$ if there is a compact set $Q$ in $G$ such that 
$G = \Lambda + Q$. A subset $\Lambda_o \subset \Lambda$ is \emph{syndetic} if there
is a finite set $F$ in $G$ such that $\Lambda \subset \Lambda_o +  F$. The following useful lemma will be used several times in the proofs of the main theorems.

\begin{lemma}
\label{Lemma_syndfromav}
Let $\Lambda \subset G$ be a set with finite local complexity, and suppose that 
$f$ is a non-negative and bounded function on $\Lambda$ with the property that for every strong F\o lner sequence $(F_n)$ in $G$, 
\[
\varliminf_{n \ra \infty} \frac{1}{m_G(F_n)} \sum_{\lambda \in \Lambda \cap F_n} f(\lambda) > 0.
\]
Then there exists $c > 0$ such that the set $\{\lambda \in \Lambda \, : \, f(\lambda) \geq c\}$
is syndetic in $\Lambda$.
\end{lemma}

\begin{proof}
Given $c > 0$, define
\[
S_c = \{ \lambda \in \Lambda \, : \, f(\lambda) \geq c \}.
\]
We first show that there exists $c > 0$ such that $S_c$ is relatively dense in $G$. To do this, fix a strong F\o lner sequence $(L_n)$ in $G$ and assume that there is no $c > 0$ for which $S_c$ is relatively dense in $G$. In particular, for every $n$, we can find an element $g_n \in G$ such that $g_n \notin S_{1/n} - L_n$, or equivalently 
$S_{1/n} \cap (L_n + g_n) = \emptyset$. By Lemma \ref{Lemma_Folner}, $F_n = L_n + g_n$ is again a strong F\o lner sequence in $G$. Now, since $S_{1/n} \cap F_n = \emptyset$, we have
\begin{align*}
\sum_{\lambda \in \Lambda \cap F_n} f(\lambda)
&= \sum_{\lambda \in (\Lambda \cap S^c_{1/n}) \cap F_n} f(\lambda) < \frac{1}{n} |\Lambda \cap F_n|,
\end{align*}
for all $n$. Since $\Lambda$ has finite local complexity, it is a $U$-uniformly discrete set in $G$ for some identity neighbourhood $U$, so by Lemma \ref{Lemma_UQV},
we have 
\[
|\Lambda \cap F_n| \leq m_G(F_n+V)/m_G(V),
\]
for every identity neighbourhood $V$ in $G$ such that $V-V \subset U$, and thus
\[
\frac{1}{m_G(F_n)} \cdot 
\sum_{\lambda \in \Lambda \cap F_n} f(\lambda) \leq \frac{m_G(F_n +V)}{n \cdot m_G(V) \cdot m_G(F_n)}, \quad \textrm{for all $n$}.
\]
Since $(F_n)$ is a strong F\o lner sequence, the right hand side tends to zero as $n \ra \infty$,
contradicting our assumption that the limit inferior of the left-hand side is strictly positive. We conclude that there exist $c > 0$ and a compact set $Q \subset G$ such that $S_c + Q = G$. In particular, $\Lambda \subset S_c + Q$. Note that if $q \in Q$
satisfies $\Lambda \cap (S_c + q) \neq \emptyset$, then $q \in \Lambda - S_c \subset \Lambda - \Lambda$, and thus
\[
\Lambda \subset S_c + Q \cap (\Lambda - \Lambda).
\]
Since $\Lambda$ has finite local complexity, $Q \cap (\Lambda - \Lambda)$ is a finite set, and thus $S_c$ is a syndetic subset of $\Lambda$.
\end{proof}

\begin{corollary}
\label{cor_syndav}
Let $\Gamma$ be a countable abelian group and suppose that $f$ is a non-negative and 
bounded function on $\Gamma$ with the property that for every F\o lner sequence $(F_n)$ in $\Gamma$, 
\[
\varliminf_{n \ra \infty} \frac{1}{|F_n|} \sum_{\gamma \in F_n} f(\gamma) > 0.
\]
Then there exists $c > 0$ such that the set $\{\gamma \in \Gamma \, : \, f(\gamma) \geq c\}$
is syndetic in $\Gamma$.
\end{corollary}

\section{Furstenberg's Correspondence Principle for uniformly discrete sets in lcsc groups}
\label{subsec:corr}

Furstenberg's Correspondence Principle is a fundamental tool in density Ramsey theory. 
In this subsection we establish a version of this principle for uniformly discrete subsets of a lcsc abelian group. The proof follows the usual route, first outlined by Furstenberg in \cite{F}, but we present it here for completeness as several steps need to be modified due to the (potential) non-discreteness of $G$. We recall that if $P_o \in \mathscr{C}(G)$ is a $U$-uniformly discrete set in $G$ and $\Omega_{P_o}^{\times}$ denotes the punctured hull of $P_o$, then the canonical cross-section $\cT_{P_o} \subset \Omega_{P_o}^{\times}$ satisfies 
\[
\Xi_{\cT_{P_o}} \subseteq \overline{P_o - P_o},
\]
by Lemma \ref{Lemma_TransSep}, and is thus $U$-separated. If $\mu \in \Prob_G(\Omega_{P_o}^{\times})$, then $\mu_{\cT_{P_o}}$ denotes the transverse measure on $\cT_{P_o}$, whose existence is guaranteed by Lemma \ref{Lemma_transmeas}.

\begin{theorem}[Transverse Correspondence Principle]
\label{Thm_CorrPrinciple}
Let $P_o \subset G$ be a uniformly discrete set with positive upper Banach density. Let $\Delta \subset G$ and suppose that $0 \in \Delta$ and there is an open identity neighbourhood $W$ in $G$
such that 
\[
(\overline{P_o - P_o} - \Delta) \cap W = \{0\}.
\]   
Then there exists an ergodic $\mu \in \Prob_G(\Omega_{P_o}^{\times})$, such that for all $r \geq 1$,
\[
d^*(P_o) = \mu_{\cT_{P_o}}(\cT_{P_o}) 
\qand 
d^*\Big(P_o \cap \Big( \bigcap_{k=1}^r (P_o - g_k) \Big)\Big)
\geq 
\mu_{\cT_{P_o}}\Big(\bigcap_{k=1}^r (-g_k).\cT_{P_o}\Big),
\]
for all $g_1,\ldots,g_r \in \Delta$.
\end{theorem}

\begin{remark}
Note that the lower bound is non-trivial only if $\{g_1,\ldots,g_r\} \subset \Xi_{\cT_{P_o}}$, as the right-hand side otherwise equals zero.
\end{remark}

The proof of Theorem \ref{Thm_CorrPrinciple} will be done in several steps. We begin by collecting a few lemmas. In what follows, let $U$ be an open and bounded identity neighbourhood in $G$ and suppose $P_o \subset G$ is a $U$-uniformly discrete set with positive upper Banach density, and let $\Delta$ and $W$ be as above. Note that we may without loss of generality assume that $W \subset U$. Fix two symmetric and open identity neighbourhoods $V_o$ and $V$ in $G$ such that
\[
V_o - V_o + V_o - V_o \subset V \subset W \qand V-V \subset U,
\]
and we require $V$ to be $m_G$-Jordan measurable. Recall from Subsection \ref{subsec:cha} that every open set in $G$ contains an open, 
symmetric and $m_G$-Jordan measurable subset, so the existence of 
such an identity neighbourhood $V$ is clear.

\begin{lemma}
\label{lemma1}
Let $V \subset G$ be a symmetric, open and $m_G$-Jordan measurable set such that $V-V \subset U$. Then, for every strong F\o lner sequence $(F_n)$ in $G$, there exists $\mu \in \Prob_G(\Omega_{P_o})$ such that $\mu(\cO_V) \geq m_G(V) \cdot \overline{d}_{(F_n)}(P_o)$.
\end{lemma}

\begin{proof}
Let $(F_n)$ be a strong F\o lner sequence in $G$ and define $\widetilde{F}_n := F_n + V$. Then, by Lemma \ref{Lemma_Folner}, $(\widetilde{F}_n)$ is again a strong F\o lner sequence in $G$. Define $\mu_n \in \Prob(\Omega_{P_o})$ by
\[
\mu_n = \frac{1}{m_G(\widetilde{F}_n)} \int_{\widetilde{F}_n} \delta_{g.P_o} \, dm_G(g).
\]
Note that
\begin{align*}
\mu_n(\cO_V)
&= \frac{1}{m_G(\widetilde{F}_n)} \int_{\widetilde{F}_n} \delta_{g.P_o}(\cO_V) \, dm_G(g) \\[0.2cm]
&= \frac{m_G(\{ g \in \widetilde{F}_n \, : \, (P_o - g) \cap V \neq \emptyset\}}{m_G(\widetilde{F}_n)} \\[0.2cm]
&= \frac{m_G((P_o + V) \cap (F_n+V))}{m_G(F_n +V)},
\end{align*}
where we in the last step have used that $V$ is symmetric. Since $P_o$ is $U$-uniformly discrete and $V-V \subset U$, 
\begin{align*}
(P_o + V) \cap (F_n+V)
&= \bigsqcup_{p \in P_o} \big((p + V) \cap (F_{n} + V)\big) \\[0.2cm]
&\supseteq \bigsqcup_{p \in P_o \cap F_{n}} (p+V),
\end{align*}
and thus
\begin{equation}
\label{OVbnd}
\mu_n(\cO_V) \geq m_G(V) \cdot \frac{|P_o \cap F_n|}{m_G(F_n)} \cdot \frac{m_G(F_n)}{m_G(F_n+V)}, \quad \textrm{for all $n$}.
\end{equation}
Pick a weak*-convergent sub-sequence $(\mu_{n_k})$ with weak*-limit $\mu$. Then $\mu$ is $G$-invariant, and since $V$ is $m_G$-Jordan measurable, $\cO_V$ is $\mu$-Jordan measurable by Lemma \ref{Lemma_Jordan}. In particular, $\mu_{n_k}(\cO_V) \ra \mu(\cO_V)$ as $k \ra \infty$, and since $(F_n)$ is assumed to be a strong F\o lner sequence in $G$, \eqref{OVbnd}, implies that
\[
\mu(\cO_V) \geq m_G(V) \cdot \overline{d}_{(F_n)}(P_o).
\]
\end{proof}

\begin{lemma}
\label{lemma2}
There exists an ergodic $\mu \in \Prob_G(\Omega_{P_o}^{\times})$ such that 
$\mu_{\cT_{P_o}}(\cT_{P_o}) \geq d^*(P_o)$.
\end{lemma}

\begin{proof}
For every $m$, pick a strong F\o lner sequence $(F_n(m))$ in $G$ such that
\[
\overline{d}_{(F_n(m))}(P_o) \geq d^*(P_o) - \frac{1}{m}.
\]
By Lemma \ref{lemma1}, we can find $\mu_m \in \Prob_G(\Omega_{P_o})$ such that 
\[
\mu_m(\cO_V) \geq m_G(V) \cdot \Big(d^*(P_o) - \frac{1}{m}\Big).
\]
We can now extract a weak*-convergent sub-sequence $(\mu_{m_k})$ with weak*-limit $\nu$. Note that $\nu$ is $G$-invariant. Since $V$ is $m_G$-Jordan measurable, $\cO_V$ is $\mu$-Jordan measurable by Lemma \ref{Lemma_Jordan} and thus $\nu(\cO_V) \geq m_G(V) \cdot d^*(P_o)$. By \cite[Theorem 8.20]{EW}, we can decompose $\nu$ into ergodic components, 
i.e. there is a Borel probability measure $\sigma$ on $\Prob_G^{\textrm{erg}}(\Omega_{P_o})$ such that
\[
\nu = \int_{\Prob_G^{\textrm{erg}}(\Omega_{P_o})} \mu \, d\sigma(\mu).
\]
In particular, we can find an ergodic $\mu \in \Prob_G(\Omega_{P_o})$ such that
\begin{equation}
\label{lwbndmu}
\mu(\cO_V) \geq \nu(\cO_V) \geq m_G(V) \cdot d^*(P_o).
\end{equation}
Since we assume that $d^*(P_o) > 0$, we conclude that $\mu \neq \delta_{\{\emptyset\}}$. In particular, since $\mu$ is ergodic and $\emptyset$ is a fixed point for the $G$-action, $\mu(\{\emptyset\}) = 0$, and thus $\mu \in \Prob_G(\Omega_{P_o}^{\times})$, so the transverse measure $\mu_{\cT_{P_o}}$ of $\mu$ is well-defined, and $\mu(\cO_V) = m_G(V) \cdot \mu_{\cT_{P_o}}(\cT_{P_o})$.
In particular, \eqref{lwbndmu} implies that $\mu_{\cT_{P_o}}(\cT_{P_o}) \geq d^*(P_o)$. 
\end{proof}

\begin{lemma}
\label{lemma3}
Let $(F_n)$ be a strong F\o lner sequence in $G$ and suppose $\mu \in \Prob_G(\Omega_{P_o}^{\times})$ is ergodic. Let $g_1,\ldots,g_r \in \Delta$. 
Then, for $\mu$-almost every $P \in \Omega_{P_o}^{\times}$, there exists a 
sub-sequence $(F_{n_j})$ such that
\[
\lim_{j \ra \infty}
\frac{m_G\big(\big(P \cap \big(\bigcap_{k=1}^r (P-g_k)\big)+V_o\big) \cap F_{n_j}\big)}{m_G(F_{n_j})}
= m_G(V_o) \cdot \mu_{\cT_{P_o}}\big(\bigcap_{k=1}^r (-g_k).\cT_{P_o}\big).
\]
\end{lemma}

\begin{proof}
Define the (possibly empty) open set $A \subset \Omega_{P_o}^{\times}$ by
\[
A := \cO_{V_o} \cap \Big( \bigcap_{k=1}^r (-g_k).\cO_{V_o} \Big).
\]
Lemma \ref{Lemma_g1gr} (applied to $B = \cT_{P_o}$, and using that $V_o$ is symmetric) implies that
\[
\mu(A) = m_G(V_o) \cdot \mu_{\cT_{P_o}}\Big(\cT_{P_o} \cap \Big(\bigcap_{k=1}^r (-g_k).\cT_{P_o}\Big)\Big).
\]
Since $\mu$ is ergodic and $(F_n)$ is F\o lner, the mean ergodic theorem tells us that
\[
\lim_{n \ra \infty} \frac{1}{m_G(F_n)} \int_{F_n} \chi_{A}(g.  \, \cdot) \, dm_G(g)
= \mu(A),
\]
in the norm topology on $L^2(\Omega_{P_o}^{\times},\mu)$. We can thus extract  a sub-sequence $(F_{n_j})$ along which the convergence is $\mu$-almost sure.
Hence, for $\mu$-almost every $P \in \Omega_{P_o}^{\times}$,
\[
\lim_{j \ra \infty} \frac{m_G(\{ g \in F_{n_j} \, : \, P - g \in A\})}{m_G(F_{n_j})} 
= \mu(A).
\] 
Note that since $V_o$ is symmetric, 
\begin{align*}
&\{ g \in F_{n_j} \, : \, P - g \in A\} \\[0.2cm]
&=
\{g \in F_{n_j} \, : \, (P - g) \cap V_o \neq \emptyset, \enskip (P-g-g_k) \cap V_o \neq \emptyset, \enskip \textrm{for all $k=1,\ldots,r$}\} \\
&=
(P+V_o) \cap \big( \bigcap_{k=1}^r (P + V_o - g_k) \big) \cap F_{n_j},
\end{align*}
for all $j$. It thus suffices to show that
\[
(P+V_o) \cap \big( \bigcap_{k=1}^r (P + V_o - g_k) \big) 
= \big(P \cap \big( \bigcap_{k=1}^r (P-g_k)\big) + V_o \big).
\]
To do this, first note the $\supseteq$-inclusion is trivial, so it is enough to prove
\begin{equation}
\label{incllr}
(P+V_o) \cap \big( \bigcap_{k=1}^r (P + V_o - g_k) \big) 
\subseteq \big(P \cap \big( \bigcap_{k=1}^r (P-g_k)\big) + V_o \big).
\end{equation}
We may without loss of generality assume that the set on the left-hand side is non-empty, and pick $g$ in this set, allowing us to write
\[
g = p_o + v_o = p_1 + v_1 - g_1 = \ldots = p_r + v_r - g_r,
\]
for some $p_o,\ldots,p_r \in P$ and $v_o,v_1,\ldots,v_r \in V_o$. Then, by Lemma \ref{Lemma_TransSep},
\[
p_k - p_o - g_k = v_o - v_k \in (P-P - \Delta) \cap (V_o-V_o) \subset (\overline{P_o - P_o} - \Delta) \cap W = \{0\},
\]
for all $k$, and thus $v_k = v_o$ and $p_o = p_k - g_k$ for all $k$. In particular, 
\[
g = p_o + v_o \in P \cap \big( \bigcap_{k=1}^r (P-g_k)\big) + V_o.
\]
Since $g$ is arbitrary, this proves \eqref{incllr}, and we are done.
\end{proof}

\begin{lemma}
\label{lemma4}
Let $P \in \Omega_{P_o}^{\times}$ and let $V_1 \subset V_o$ be an open identity neighbourhood. Then, for all $g_1,\ldots,g_r \in \Delta$ and for every 
pre-compact set $Q \subset G$, there exists $h \in G$ such that
\[
\big(P \cap \big( \bigcap_{k=1}^r (P-g_k)\big) + V_o\big) \cap Q
\subseteq P_o \cap \big( \bigcap_{k=1}^r (P_o-g_k)\big) - h + V_o + V_1.
\]
\end{lemma}

\begin{proof}
Set $g_o = 0$, and define
\[
P' := \bigcap_{k=0}^r (P-g_k) \qand P_o' := \bigcap_{k=0}^r (P_o - g_k).
\]
Fix a pre-compact set $Q \subset G$. If $(P' + V_o) \cap Q$ is empty, the inclusion in the lemma is trivial, so we assume that this set is non-empty. Since $P$ is $U$-uniformly discrete, we can write
\[
(P' + V_o) \cap Q = \bigsqcup_{j=1}^q \big((p_j + V_o) \cap Q\big)
\]
for some $q \geq 1$ and $p_1,\ldots,p_q \in P'$. Furthermore, 
for all $k = 0,\ldots,r$, we can find elements $p_{j,k} \in P$ such that 
$p_j = p_{j,k} - g_k$ for all $j=1,\ldots,q$. Let us now fix a sequence $(h_n)$ in $G$
such that $h_n.P_o = P_o - h_n \ra P$ as $n \ra \infty$. Then, by (ii) in the characterization of sequential convergence in the Chabauty topology in Subsection \ref{subsec:cha}, we 
can find, for all $j$ and $k$, an integer $n_{j,k}$ such that for all $n \geq n_{j,k}$,
there exist elements $p^o_{j,k}(n) \in P_o$ and $v_{j,k}(n) \in V_1$ such that
\[
p_{j,k} = p^o_{j,k}(n) - h_n + v_{j,k}(n).
\]
Let $n_o := \max_{j,k} n_{j,k}$. Then, for all $n \geq n_o$ and $v_o,\ldots,v_r \in V_o$,
\[
p_j + v_j = p_{j,k} - g_k + v_j \in P_o - (g_k+h_n) + V_o + V_1 , \quad \textrm{for all $k=0,1,\ldots,r$},
\]
and thus (since $v_o,\ldots,v_r \in V_o$ are arbitrary), 
\[
(P' + V_o) \cap Q \subset \bigcap_{k=0}^r \big(P_o - g_k + V_o + V_1\big) - h_n, \quad \textrm{for all $n \geq n_o$}.
\]
Let $h = h_n$ for some $n \geq n_o$. Then to prove the lemma, it suffices to show that
\[
\bigcap_{k=0}^r \big(P_o - g_k + V_o + V_1\big) = \bigcap_{k=0}^r
\big(P_o - g_k\big) + V_o + V_1.
\]
The $\supseteq$-inclusion is trivial, so we only need to establish the other inclusion. 
To do this, pick $g$ in the set on the left-hand side, and write
\[
g = p_k - g_k + v_k, \quad \textrm{for $k=0,\ldots,r$},
\]
and $p_o,\ldots,p_r \in P_o$ and $v_0,\ldots,v_r \in V_o + V_1$. Hence, since $V_1 \subset V_o$ and $V_o + V_o - V_o - V_o \subset W$, it follows from Lemma \ref{Lemma_TransSep},
\[
p_k - p_o - g_k = v_o - v_k \in (P - P - \Delta) \cap (V_o + V_1 - V_o - V_1)
\subset (\overline{P_o - P_o} - \Delta) \cap W = \{0\},
\]
for all $k=1,\ldots,r$. We conclude that $v_k = v_o$ and $p_o = p_k - g_k$ for all 
$k$, and thus $p_o \in P_o'$ and $g = p_o + v_o \in P_o' + V_o + V_1$, which finishes the proof.
\end{proof}

\subsubsection*{\textbf{\emph{Proof of Theorem \ref{Thm_CorrPrinciple}}}}

Let the sets $P_o, \Delta, W, V_o, V$ and $U$ be as above. By Lemma \ref{lemma2}, there is at least one ergodic $\mu \in \Prob_G(\Omega_{P_o}^{\times})$ such that $\mu_{\cT_{P_o}}(\cT_{P_o}) \geq d^*(P_o)$. We want to show that
\begin{equation}
\label{wow}
d^*\Big(P_o \cap \Big( \bigcap_{k=1}^r (P_o - g_k) \Big)\Big)
\geq 
\mu_{\cT_{P_o}}\Big(\bigcap_{k=1}^r (-g_k).\cT_{P_o}\Big),
\end{equation}
for all $g_1,\ldots,g_r \in \Delta$. In particular, since $0 \in \Delta$, we
conclude from \eqref{wow} that $d^*(P_o) \geq \mu_{\cT_{P_o}}(\cT_{P_o})$, and thus $d^*(P_o) = \mu_{\cT_{P_o}}(\cT_{P_o})$. \\

To prove \eqref{wow}, fix $g_1,\ldots,g_r \in \Delta$. Moreover, fix a strong F\o lner sequence $(F_n)$
in $G$ and an open identity neighbourhood $V_1 \subset V_o$. Define $\widetilde{F}_n = F_n + V_o + V_1$, and note that $(\widetilde{F}_n)$ is a strong F\o lner sequence by Lemma \ref{Lemma_Folner}. Hence, by Lemma \ref{lemma3}, we can find a point $P \in \Omega_{P_o}^{\times}$ and a sub-sequence $(\widetilde{F}_{n_k})$ such that
\[
\lim_{j \ra \infty}
\frac{m_G\big(\big(P \cap \big(\bigcap_{k=1}^r (P-g_k)\big)+V_o\big) \cap \widetilde{F}_{n_j}\big)}{m_G(\widetilde{F}_{n_j})}
= m_G(V_o) \cdot \mu_{\cT_{P_o}}\big(\bigcap_{k=1}^r (-g_k).\cT_{P_o}\big).
\]
For every $j$, we can apply Lemma \ref{lemma4} to find $h_j \in G$ such that
\begin{align*}
&\frac{m_G\big(\big(P \cap \big(\bigcap_{k=1}^r (P-g_k)\big)+V_o\big) \cap \widetilde{F}_{n_j}\big)}{m_G(\widetilde{F}_{n_j})} \\[0.2cm]
&\leq \frac{m_G\big(\big(P_o \cap \big(\bigcap_{k=1}^r (P_o-g_k)\big) - h_j + V_o + V_1 \big) \cap \widetilde{F}_{n_j}\big)}{m_G(\widetilde{F}_{n_j})} \\[0.2cm]
&= 
\frac{m_G\big(\big(P_o \cap \big(\bigcap_{k=1}^r (P_o-g_k)\big)+V_o + V_1 \big) \cap (F_{n_j} + V_o + V_1 + h_j) \big)}{m_G(F_{n_j} + V_o + V_1 + h_j)}.
\end{align*}
Since 
\[
(V_o + V_1) - (V_o + V_1) \subset V_o + V_o - V_o - V_o \subset V,
\]
Lemma \ref{Lemma_UQV} (applied to the $U$-uniformly discrete set $P_o \cap \big(\bigcap_{k=1}^r (P_o-g_k)\big)$ and the bounded set $Q = F_{n_j} + h_j$) yields, 
\begin{align*}
&m_G\big(\big(P_o \cap \big(\bigcap_{k=1}^r (P_o-g_k)\big)+V_o + V_1 \big) \cap (F_{n_j} + V_o + V_1 + h_j) \big) \\[0.2cm]
&\leq m_G(V_o + V_1) \cdot |P_o \cap \big(\bigcap_{k=1}^r (P_o-g_k)\big) \cap (F_{n_j} + h_j + V)|.
\end{align*}
Set $L_j = F_{n_j} + h_j + V$, and note that Lemma \ref{Lemma_Folner} again guarantees that $(L_j)$ is a strong F\o lner sequence in $G$. Then,
\begin{align*}
&\mu_{\cT_{P_o}}\big(\bigcap_{k=1}^r (-g_k).\cT_{P_o}\big)
\\[0.2cm]
&\leq \varlimsup_{j \ra \infty} \frac{m_G(V_o + V_1)}{m_G(V_o)} \cdot \frac{|P_o \cap \big(\bigcap_{k=1}^r (P_o-g_k)\big) \cap L_j|}{m_G(L_j)} \cdot \frac{m_G(L_j)}{m_G(F_{n_j} + V_o + V_1 + h_j)} \\[0.2cm]
&\leq \frac{m_G(V_o + V_1)}{m_G(V_o)} \cdot \overline{d}_{(L_j)}(P_o \cap \big(\bigcap_{k=1}^r (P_o-g_k)\big)) \leq 
\frac{m_G(V_o + V_1)}{m_G(V_o)} \cdot d^*(P_o \cap \big(\bigcap_{k=1}^r (P_o-g_k)\big)),
\end{align*}
where we in the second to last step have used that $(L_j)$ is a strong F\o lner sequence (and thus the last factor on the middle line tends to $1$ as $j \ra \infty$).
Note that this inequality holds for any open set $V_1 \subset V_o$. In particular, the quotient $\frac{m_G(V_o + V_1)}{m_G(V_o)}$ can be made arbitrarily close to $1$ by taking $V_1$ small enough. We conclude that 
\[
\mu_{\cT_{P_o}}\big(\bigcap_{k=1}^r (-g_k).\cT_{P_o}\big) \leq  d^*(P_o \cap \big(\bigcap_{k=1}^r (P_o-g_k)\big)), 
\]
and thus \eqref{wow} is proved.

\section{Furstenberg-Katznelson Theorem for lcsc abelian groups}
\label{Sec:FK}
Let $G$ be a locally compact and second countable (lcsc) abelian group. If $(Z,\theta)$
is a standard Borel probability measure space and $a, b : G \times Z \ra Z$ are two Borel $G$-actions on $Z$, we say that they \emph{commute} if
\[
a(g_1,b(g_2,z)) = b(g_2,a(g_1,z)), \quad \textrm{for all $g_1,g_2 \in G$ and $z \in Z$}.
\]
For countable abelian groups, the following theorem is due to Furstenberg and 
Katznelson (\cite[Theorem B]{FK1} and \cite[Secion 3]{FK2}) and Austin 
\cite[Theorem B]{austin}. We show below that the general case
follows from a simple approximation argument. 

\begin{theorem}
\label{Thm_SyndMultRecG}
Let $(Z,\theta)$ be a standard Borel probability measure space and let $a_1,\ldots,a_r$ be commuting 
$\theta$-preserving Borel $G$-actions on $Z$. Then, for every $\theta$-measurable set $A \subset Z$
with positive $\theta$-measure, there exists $c > 0$ such that the set
\[
S := \Big\{ g \in G \, : \, 
\theta\Big(A \cap \Big( \bigcap_{k=1}^r a_k(g,A) \Big)\Big) \geq c \Big\}
\]
is syndetic in $G$. In particular, for every strong F\o lner sequence $(F_n)$ in $G$,
\[
\varliminf_{n \ra \infty} \frac{1}{m_G(F_n)} \int_{F_n} \theta\Big(A \cap \Big( \bigcap_{k=1}^r a_k(g,A) \Big)\Big) \, dm_G(g) > 0.
\]
\end{theorem}

\subsection{Proof of Theorem \ref{Thm_SyndMultRecG}}

Let $a_1,\ldots,a_r$ be commuting $\theta$-preserving Borel $G$-actions on $Z$, and 
let $A \subset Z$ be a $\theta$-measurable set with positive $\theta$-measure. Since
$G$ is locally compact and second countable, it is separable, and thus we can find 
a dense countable subgroup $\Gamma < G$. The following theorem is due to Austin \cite[Theorem B]{austin}.

\begin{theorem}
\label{Thm_austin}
Let $(Z,\theta)$ be a standard Borel probability measure space and let $a_1,\ldots,a_r$ be commuting $\theta$-preserving Borel $\Gamma$-actions on $Z$. Then, for every $\theta$-measurable set $A \subset Z$ with positive $\theta$-measure and for every F\o lner sequence $(F_n)$ in $\Gamma$, 
\[
\varliminf_{n \ra \infty} \frac{1}{|F_n|} \sum_{\gamma \in F_n} 
\theta\Big(A \cap \Big( \bigcap_{k=1}^r a_k(\gamma,A) \Big)\Big)  > 0.
\]
\end{theorem}

\noindent 
Define $\varphi : G \ra [0,1]$ by
\[
\varphi(g) = \theta\Big(A \cap \Big( \bigcap_{k=1}^r a_k(g,A) \Big)\Big), \quad g \in G.
\] 
By Theorem \ref{Thm_austin} and Corollary \ref{cor_syndav} (applied to the function $f = \varphi|_\Gamma$), we conclude that there exists 
$c_o > 0$ such that
\[
S_\Gamma = \Big\{ \gamma \in \Gamma \, : \, \varphi(\gamma) \geq c_o \Big\}
\]
is syndetic in $\Gamma$. The following simple lemma will be proved below.

\begin{lemma}
\label{lemma_unifcont}
$\varphi$ is uniformly continuous on $G$.
\end{lemma}

Taking the lemma for granted for now, we can thus find an open set $V \subset G$ such that 
\begin{equation}
\label{varphig1g2}
|\varphi(g_1) - \varphi(g_2)| < c_o/2, \quad  
\textrm{for all $g_1,g_2 \in G$ such that $g_1 - g_2 \in V$}. 
\end{equation}
Define $S_o = S_\Gamma + V \subset G$. Since $S_\Gamma$ is syndetic in $\Gamma$, there is a finite set $F \subset \Gamma$ such that $S_\Gamma + F = \Gamma$, 
and thus
\[
S_o + F = (S_\Gamma + F) + V = \Gamma + V = G,
\]
so $S_o$ is syndetic in $G$. Now if $g \in S_o$, we can clearly write $g = \gamma + v$ for some $\gamma \in S_\Gamma$ and $v \in V$, and thus by \eqref{varphig1g2} (with $g_1 = \gamma + v$ and $g_2 = \gamma$),
\[
\varphi(g) = \varphi(\gamma + v) - \varphi(\gamma) + \varphi(\gamma) \geq \varphi(\gamma) - c_o/2 \geq c_o/2.
\]
Let $c = c_o/2$ and define
\[
S = \{ g \in G \, : \, \varphi(g) \geq c \}.
\]
The argument above shows that $S_o \subset S$ and thus $S$ is syndetic as well, finishing the proof of Theorem \ref{Thm_SyndMultRecG}.

\subsection{Proof of Lemma \ref{lemma_unifcont}}

By \cite[Lemme A.1.1]{anetal}, the induced $L^1$-representations associated with the actions $a_1,\ldots,a_r$ are strongly continuous, so for every $\eps > 0$, we can 
find an open identity neighbourhood $U \subset G$ such that
\begin{equation}
\label{akeps}
\int_Z |\chi_A(z) - \chi_{A}(a_k(g,z)) | \, d\theta(z) < \eps/r, \quad \textrm{for all $g \in U$ and $k=1,\ldots,r$}.
\end{equation}
Note that
\[
\varphi(g) 
= \int_Z \chi_A(z) \, \prod_{k=1}^r \chi_A(a_k(-g,z)) \, d\theta(z), \quad \textrm{for all $g \in G$}.
\]
A simple induction argument shows that for all $g_1, g_2 \in G$, 
\[
\varphi(g_1) - \varphi(g_2)
= \sum_{k=1}^r \int_Z \beta_k(z) \cdot (\chi_{A}(a_k(-g_1,z) - \chi_{A}(a_k(-g_2,z)) \, d\eta(z),
\]
where
\[
\beta_k(z) = \big(\prod_{j=k+1}^{r} \chi_{A}(a_j(-g_1,z))\big) \cdot 
\big(\prod_{j=1}^{k-1} \chi_{A}(a_j(-g_2,z))\big).
\]
Here we use the convention that products over empty sets are always equal to $1$. Note that $\|\beta_k\|_\infty \leq 1$ for all $k$. \\

Fix $\eps > 0$ and let $U$ be as above. Note that if $g_1 - g_2 \in U$, then by \eqref{akeps},
\begin{align*}
|\varphi(g_1) - \varphi(g_2)| 
&\leq \sum_{k=1}^r \int_Z |\chi_A(a_k(-g_1,z)) - \chi_{A}(a_k(-g_2,z)) | \, d\theta(z) \\[0.2cm]
&= 
\sum_{k=1}^r \int_Z |\chi_A(z) - \chi_{A}(a_k(g_1-g_2,z)) | \, d\theta(z)
< \eps.
\end{align*}
Since $\eps > 0$ and $g_1,g_2 \in G$ such that $g_1 - g_2 \in U$ are arbitrary, $\varphi$ is uniformly continuous.

\section{Proof of Theorem \ref{Thm_MainErg}}
\label{sec:prfdyn}

Let $(X,\mathscr{B}_X)$ be a Borel $G$-space, equipped with a $U$-separated  cross-section $Y \subset X$ for some identity neighbourhood $U$ in $G$. Furthermore, 
let $\Lambda$ and $\Delta$ be uniformly discrete subsets of $G$ and suppose that
there exist $\alpha_1,\ldots,\alpha_r \in \End(G)$ such that $\alpha_k(\Lambda) \subseteq \Delta$ for all $k=1,\ldots,r$. \\

We further assume that $\Lambda$ has finite local complexity and that there is a set $\Lambda_o$ with positive upper Banach density such that $\Lambda_o - \Lambda_o \subseteq \Lambda$. We also assume there is an identity neighbourhood $V \subset U$ such that $(\Xi_Y - \Delta) \cap V = \{0\}$. \\

In this section, we show the following slightly stronger version of Theorem \ref{Thm_MainErg} (which corresponds to the case when $\Delta = \Lambda^q$). 

\begin{theorem}
\label{Thm_SyndLevelSet}
For every $\mu \in \Prob_G(X)$ and Borel set $B \subset Y$
with positive $\mu_Y$-measure, there exists $c > 0$ such that the set
\[
S = \Big\{ \lambda \in \Lambda \, : \, \mu_Y\Big( \bigcap_{k=1}^r (-\alpha_k(\lambda)).B\Big) \geq c \Big\}
\]
is syndetic in $\Lambda$.
\end{theorem}

Since $\Lambda$ is assumed to have finite local complexity in $G$, Lemma \ref{Lemma_syndfromav} tells us that Theorem \ref{Thm_SyndLevelSet} is an immediate consequence of the following theorem, whose proof will be presented below.

\begin{theorem}
\label{Thm_PosAvLambda}
Let $\Delta, \Lambda$ and $\alpha_1,\ldots,\alpha_r$ be as in Theorem \ref{Thm_SyndLevelSet}. Let $\mu \in \Prob_G(X)$ and let $B \subset Y$
be a Borel set with positive $\mu_Y$-measure. Then, for every strong F\o lner 
sequence $(F_n)$ in $G$, 
\[
\varliminf_{n \ra \infty}
\frac{1}{m_G(F_n)} \sum_{\lambda \in \Lambda \cap F_n} \mu_Y\Big( \bigcap_{k=1}^r (-\alpha_k(\lambda)).B\Big) > 0.
\]
\end{theorem}

\subsection{Proof of Theorem \ref{Thm_PosAvLambda}}

In what follows, let us fix $\mu \in \Prob_G(X)$ and a Borel set $B \subset Y$
with positive $\mu_Y$-measure. We first show how Theorem \ref{Thm_PosAvLambda}
can be deduced from the following proposition. 

\begin{proposition}
\label{Prop_LowBndAvLambda}
Let $\Delta, \Lambda, \Lambda_o$ and $\alpha_1,\ldots,\alpha_r$ be as in Theorem \ref{Thm_SyndLevelSet}. Let $\nu \in \Prob_G(\Omega_{\Lambda_o}^{\times})$ and let $(F_n)$ be a strong F\o lner sequence in $G$. Then there exist $\beta_o > 0$, a non-negative null sequence $(\delta_n)$, $r$ commuting $\mu \otimes \nu$-preserving Borel $G$-actions $a_1,\ldots,a_r$ on $X \times \Omega_{\Lambda_o}^{\times}$ and a Borel set $A_o \subset X \times \Omega_{\Lambda_o}^{\times}$
with positive $\mu \otimes \nu$-measure such that
\begin{align*}
&\sum_{\lambda \in \Lambda \cap F_n} \mu_Y\Big( \bigcap_{k=1}^r (-\alpha_k(\lambda)).B\Big) \\[0.1cm]
&\geq \beta_o \cdot \int_{F_n} (\mu \otimes \nu)\Big(A_o \cap \Big( \bigcap_{k=1}^r a_k(g,A_o) \Big) \Big) \, dm_G(g) - \delta_n \cdot m_G(F_n), \quad \textrm{for all $n$}.
\end{align*}
\end{proposition}

\begin{proof}[Proof of Theorem \ref{Thm_PosAvLambda} assuming Proposition \ref{Prop_LowBndAvLambda}]
Let $(F_n)$ be a strong F\o lner sequence in $G$. Then Proposition \ref{Prop_LowBndAvLambda} tells us that there exist $\beta_o > 0$,
$r$ commuting $\mu \otimes \nu$-preserving Borel $G$-actions $a_1,\ldots,a_r$ on $X \times \Omega_{\Lambda_o}^{\times}$ and a Borel set $A_o \subset X \times \Omega_{\Lambda_o}^{\times}$
with positive $\mu \otimes \nu$-measure such that
\begin{align*}
&\varliminf_{n \ra \infty} \frac{1}{m_G(F_n)} \cdot \sum_{\lambda \in \Lambda \cap F_n}
\mu_Y\Big( \bigcap_{k=1}^r (-\alpha_k(\lambda)).B\Big)  \\[0.2cm]
&\geq \beta_o \cdot \varliminf_{n \ra \infty} \frac{1}{m_G(F_n)} \cdot
\int_{F_n} (\mu \otimes \nu)\Big(A_o \cap \Big( \bigcap_{k=1}^r a_k(g,A_o) \Big) \Big) \, dm_G(g).
\end{align*}
The positivity of the last expression now follows from Theorem \ref{Thm_SyndMultRecG},
applied to the Borel space $(Z,\theta) = (X \times \Omega_{\Lambda_o},\mu \otimes \nu)$.
\end{proof}

\subsection{Proof of Proposition \ref{Prop_LowBndAvLambda}}

\noindent Before we begin the proof, let us briefly recall the main players in this proposition.
\vspace{0.2cm}
\begin{itemize}
\item[(i)] $(X,\mathscr{B}_X)$ is a Borel $G$-space, equipped with a $U$-separated cross-section $Y$ for some identity neighbourhood $U$ in $G$. Furthermore, we fix $\mu \in \Prob_G(X)$ and a Borel set $B \subset Y$ with positive $\mu_Y$-measure. \vspace{0.2cm}
\item[(ii)] $\Lambda_o \subset G$ is a uniformly discrete set with positive upper Banach density, and we fix $\nu \in \Prob_G(\Omega_{\Lambda_o}^{\times})$. Note that
$\nu(\cO_{W}) > 0$ for all identity neighbourhoods $W$ in $G$, where $\cO_W = (-W).\cT_{\Lambda_o} \subset \Omega_{\Lambda_o}^{\times}$. \vspace{0.2cm}
\item[(iii)] $\Lambda$ and $\Delta$ are uniformly discrete sets in $G$ such that 
\[
\Lambda_o - \Lambda_o \subset \Lambda \qand (\Xi_Y - \Delta) \cap V = \{0\}
\]
for some identity neighbourhood $V \subset U$. \vspace{0.2cm}
\item[(iv)] Fix $\alpha_1,\ldots,\alpha_r \in \End(G)$ such that $\alpha_k(\Lambda) \subset \Delta$ for all $k=1,\ldots,r$. \vspace{0.2cm}
\item[(v)] Fix a strong F\o lner sequence $(F_n)$ in $G$.
\end{itemize}
\vspace{0.2cm}

\noindent In what follows, we fix an identity neighbourhood $V_o$ such that $V_o - V_o \subset V$. By (iv), we have $\alpha_k(\Lambda) \subseteq \Delta$ for all $k$, so Lemma \ref{Lemma_g1gr} (applied to $g_k = \alpha_k(\lambda)$ for $\lambda \in \Lambda) $ tells us that
\[
\mu_Y\Big( B \cap \bigcap_{k=1}^r (-\alpha_k(\lambda)).B\Big)
= \frac{1}{m_G(V_o)} \cdot \mu\Big(B_{V_o} \cap \Big( \bigcap_{k=1}^r (-\alpha_k(\lambda)).B_{V_o}\Big)\Big), \quad \textrm{for all $\lambda \in \Lambda$}.
\]
In particular,
\begin{equation}
\label{muYmu}
\sum_{\lambda \in \Lambda \cap F_n} \mu_Y\Big( B \cap \bigcap_{k=1}^r (-\alpha_k(\lambda)).B\Big)
=
\frac{1}{m_G(V_o)} \cdot \sum_{\lambda \in \Lambda \cap F_n} 
\mu\Big(B_{V_o} \cap \Big( \bigcap_{k=1}^r (-\alpha_k(\lambda)).B_{V_o}\Big)\Big)
\end{equation}
for all $n$. Since $\alpha_1,\ldots,\alpha_r$ are continuous endomorphisms of $G$, we can find a \emph{symmetric} identity neighbourhood $V_1 \subset V_o$ such that $V_1 + \alpha_k(V_1) \subset V_o$ for all $k$. In particular,
\[
B_{V_1} \subset B_{V_o} \qand \alpha_k(h).B_{V_1} \subset B_{V_o}, \quad \textrm{for all $h \in V_1$ and $k$},
\]
and thus
\begin{equation}
\label{akgh}
\mu\Big(B_{V_o} \cap \Big( \bigcap_{k=1}^r (-\alpha_k(g)).B_{V_o}\Big)\Big)
\geq
\mu\Big(B_{V_1} \cap \Big( \bigcap_{k=1}^r (-\alpha_k(g-h)).B_{V_1}\Big)\Big), 
\end{equation}
for all $g \in G$ and $h \in V_1$. Let us now define
\[
F_n' := \bigcap_{h \in V_1} (F_n - h) \enskip \textrm{and note that $F_n' + V_1 \subset F_n$}.
\]
Hence, 
\[
\chi_{F_n}(g) \geq 
\frac{1}{m_G(V_1)} \int_G \chi_{V_1}(h) \chi_{F_n'}(g-h) \, dm_G(h) \quad \textrm{for all $g \in G$},
\]
and it follows from \eqref{akgh} that,
\begin{align*}
&\int_G \chi_{V_1}(h) \chi_{F_n'}(g-h) \, dm_G(h) \cdot \mu\Big(B_{V_o} \cap \Big( \bigcap_{k=1}^r (-\alpha_k(g)).B_{V_o}\Big)\Big) \\
&\geq 
\int_G \chi_{V_1}(h) \chi_{F_n'}(g-h) \, \mu\Big(B_{V_1} \cap \Big( \bigcap_{k=1}^r (-\alpha_k(g-h)).B_{V_1}\Big)\Big) \, dm_G(h), \quad \textrm{for all $g \in G$}.
\end{align*}
Let us now put of all this together. It follows from the last inequality above that
\begin{align*}
&\sum_{\lambda \in \Lambda \cap F_n} 
\mu\Big(B_{V_o} \cap \Big( \bigcap_{k=1}^r (-\alpha_k(\lambda)).B_{V_o}\Big)\Big) 
= \sum_{\lambda \in \Lambda} \chi_{F_n}(\lambda) \cdot \mu\Big(B_{V_o} \cap \Big( \bigcap_{k=1}^r (-\alpha_k(\lambda)).B_{V_o}\Big)\Big) 
\\[0.2cm]
&\geq \frac{1}{m_G(V_1)} \cdot \sum_{\lambda \in \Lambda} \int_G \chi_{V_1}(h) \chi_{F_n'}(\lambda-h) \, dm_G(h) \cdot \mu\Big(B_{V_o} \cap \Big( \bigcap_{k=1}^r (-\alpha_k(\lambda)).B_{V_o}\Big)\Big)
\\[0.1cm]
&\geq \frac{1}{m_G(V_1)} \cdot  \sum_{\lambda \in \Lambda} \int_G \chi_{V_1}(h) \chi_{F_n'}(\lambda-h) \mu\Big(B_{V_1} \cap \Big( \bigcap_{k=1}^r (-\alpha_k(\lambda-h)).B_{V_1}\Big)\Big)\, dm_G(h) \\[0.1cm]
&= 
\frac{1}{m_G(V_1)} \cdot \sum_{\lambda \in \Lambda} \int_G \chi_{V_1}(\lambda-h) \chi_{F_n'}(h) \mu\Big(B_{V_1} \cap \Big( \bigcap_{k=1}^r (-\alpha_k(h)).B_{V_1}\Big)\Big)\, dm_G(h) \\[0.1cm]
&= \frac{1}{m_G(V_1)} \cdot \int_{F_n'} \mu\Big(B_{V_1} \cap \Big( \bigcap_{k=1}^r (-\alpha_k(h)).B_{V_1}\Big)\Big) \cdot \Big( \sum_{\lambda \in \Lambda} \chi_{V_1}(\lambda-h) \Big) \, dm_G(h),
\end{align*}
where we in the last step used Fubini's Theorem to interchange summation and integration. Let us now fix an identity neigbourhood $V_2$ in $G$ such that $V_2 - V_2 \subset V_1$. Since both $V_1$ and $\Lambda$ are symmetric, Lemma \ref{Lemma_lwbndnu} tells us that
\[
\sum_{\lambda \in \Lambda} \chi_{V_1}(\lambda-h) = \sum_{\lambda \in \Lambda} \chi_{V_1}(h+\lambda)
\geq \nu(\cO_{V_2} \cap h.\cO_{V_2}), \quad \textrm{for all $h \in G$},
\]
where $\nu$ is defined in (ii) above. Note that $\nu(\cO_{V_2}) > 0$. We have now shown that
\begin{align}
&\sum_{\lambda \in \Lambda \cap F_n} 
\mu\Big(B_{V_o} \cap \Big( \bigcap_{k=1}^r (-\alpha_k(\lambda)).B_{V_o}\Big)\Big) \nonumber \\[0.2cm]
&\geq 
\frac{1}{m_G(V_1)} \cdot  \int_{F_n'} \mu\Big(B_{V_1} \cap \Big( \bigcap_{k=1}^r (-\alpha_k(h)).B_{V_1}\Big)\Big) \cdot \nu(\cO_{V_2} \cap h.\cO_{V_2}) \, dm_G(h). \label{V1}
\end{align}
Let us now rewrite the integrand. First, let $Z = X \times \Omega_{\Lambda_o}^{\times}$
and
\[
A_o = B_{V_1} \times \cO_{V_2} \subset Z \qand \theta = \mu \otimes \nu \in \Prob_G(Z).
\]
Then, $\theta(A_o) = \mu(B_{V_1}) \cdot \nu(\cO_{V_2}) = m_G(V_1) \mu_Y(B) \nu(\cO_{V_2}) > 0$. \\

\noindent Furthermore, we define the Borel $G$-actions $a_k : G \times Z \ra Z$ by
\[
a_1(g,(x,\Lambda_o')) = (-\alpha_1(g).x,g.\Lambda_o') 
\]
and
\[
a_k(g,(x,\Lambda_o')) = (-\alpha_k(g).x,\Lambda_o'), \quad k = 2,\ldots,r,
\]
for $(x,\Lambda_o') \in Z$. 
Clearly these actions preserve $\theta$, they all commute with each other, and
\[
\mu\Big(B_{V_1} \cap \Big( \bigcap_{k=1}^r (-\alpha_k(g)).B_{V_1}\Big)\Big) \cdot \nu(\cO_{V_2} \cap g.\cO_{V_2})
= (\mu \otimes \nu)\Big( A_o \cap \Big( \bigcap_{k=1}^r a_k(g,A_o) \Big)\Big),
\]
for all $g \in G$. We now conclude from \eqref{muYmu} and \eqref{V1} that
\begin{align*}
&\sum_{\lambda \in \Lambda \cap F_n} \mu_Y\Big( \bigcap_{k=1}^r (-\alpha_k(\lambda)).B\Big)
\geq \\[0.2cm]
&\frac{1}{m_G(V_o)m_G(V_1)} \cdot \int_{F_n'} (\mu \otimes \nu)\Big( A_o \cap \Big( \bigcap_{k=1}^r a_k(g,A_o) \Big)\Big) \, dm_G(g)
\end{align*}
for all $n$. Since both $\mu$ and $\nu$ are probability measures, 
\begin{align*}
&\int_{F_n'} (\mu \otimes \nu)\Big( A_o \cap \Big( \bigcap_{k=1}^r a_k(g,A_o) \Big)\Big) \, dm_G(g)
\geq \\[0.2cm]
&\int_{F_n} (\mu \otimes \nu)\Big( A_o \cap \Big( \bigcap_{k=1}^r a_k(g,A_o) \Big)\Big) \, dm_G(g) - m_G(F_n \setminus F_n'),
\end{align*}
for all $n$. Hence, if we set
\[
\beta_o = \frac{1}{m_G(V_o)m_G(V_1)} \qand \delta_n = \frac{\beta_o \cdot m_G(F_n \setminus F_n')}{m_G(F_n)},
\]
then
\begin{align*}
&\sum_{\lambda \in \Lambda \cap F_n} \mu_Y\Big( \bigcap_{k=1}^r (-\alpha_k(\lambda)).B\Big) \geq \\[0.2cm]
&\beta_o \cdot \int_{F_n} (\mu \otimes \nu)\Big( A_o \cap \Big( \bigcap_{k=1}^r a_k(g,A_o) \Big)\Big) \, dm_G(g) - \delta_n \cdot m_G(F_n),
\end{align*}
for all $n$. Since $(F_n)$ is a strong F\o lner sequence, $(\delta_n)$ is a null sequence, and we are done.

\section{Proof of Theorem \ref{Thm_MainComb}}
\label{sec:prfcomb}
Let $\Lambda, \Lambda_o, P_o,q$ and $\alpha_1,\ldots,\alpha_r$ be as in Theorem \ref{Thm_MainComb}. Set $\Delta = \Lambda^q$, so that $\alpha_k(\Lambda) \subseteq \Delta$ for all $k$. Furthermore, by Lemma \ref{Lemma_TransSep}, 
\[
\Xi_{\cT_{P_o}} - \Delta \subset \overline{P_o - P_o} - \Lambda^q \subset \Lambda - \Lambda - \Lambda^q \subset \Lambda^{q+2},
\]
which is uniformly discrete in $G$ since $\Lambda$ is an approximate lattice. We conclude that the conditions on $\Lambda$ and $\Delta$ in Theorem \ref{Thm_MainErg} are satisfied. \\

Since $d^*(P_o) > 0$, Theorem \ref{Thm_CorrPrinciple} further tells us that there exists $\mu \in \Prob_{G}(\Omega^{\times}_{P_o})$ such that $\mu_{\cT_{P_o}}(\cT_{P_o}) > 0$ and
\begin{equation}
\label{dmu}
d^*\Big( P_o \cap \Big( \bigcap_{k=1}^r (P_o - \alpha_k(\lambda)) \Big) \Big) 
 \geq \mu_{\cT_{P_o}}\Big(\bigcap_{k=1}^r (-\alpha_k(\lambda)).\cT_{P_o}\Big),
\end{equation}
for all $\lambda \in \Lambda$. Hence, by Theorem \ref{Thm_MainErg}, 
applied to $X = \Omega_{P_o}^{\times}, Y = \cT_{P_o}$ and $B = \cT_{P_o}$,
there exists $c > 0$ such that
\[
S_o = \big\{ \lambda \in \Lambda \, : \, 
\mu_{\cT_{P_o}}\big(\bigcap_{k=1}^r (-\alpha_k(\lambda)).\cT_{P_o}\big) \geq c\big\}
\]
is syndetic in $\Lambda$. By \eqref{dmu}, we have $S_o \subset S$,
where
\[
S= \Big\{ \lambda \in \Lambda \, : \,  d^*\Big( P_o \cap \Big( \bigcap_{k=1}^r (P_o - \alpha_k(\lambda)) \Big) \Big)
\geq c\Big\},
\]
and thus $S$ is syndetic in $\Lambda$ as well.
\appendix

\section{Szemer\'edi's Theorem for cut-and-project sets in Euclidean spaces via the IP-Szemer\'edi Theorem}

The aim of this section is to prove the following variation of Theorem \ref{Thm_R}
using the deep ergodic $\IP$-theory developed by Furstenberg and Katznelson in \cite{FK2}, and further extended by Bergelson and McCutcheon in \cite{BM2}. 

\begin{theorem}
\label{Thm_Rext}
Let $\Lambda \subset \bR^d$ be an approximate lattice and suppose that 
\[
\Delta = \Delta(\bR,\bR^m,\Gamma,W) \subset \bR 
\]
is a cut-and-project set for some lattice $\Gamma < \bR \times \bR^m$ which projects injectively to 
$\bR$ and densely to $\bR^m$, and for some bounded Borel set $W \subset \bR^m$ whose interior
is a convex and symmetric set containing $0$. Assume that $\delta \cdot \Lambda \subset \Lambda$ for every $\delta \in \Delta$. Then, for every $P_o \subset \Lambda$
with positive upper Banach density and finite subset $F \subset \Lambda^\infty$, 
there exist an integer $n$ and a syndetic and symmetric subset $\Delta_n \subset \Delta$ such that
$\Delta_n^n \subset \Delta$ with the property that for every  $\delta \in \Delta_n$,
there exist an integer $j = 1,\ldots,n$ and $p_o \in P_o$, such that
\[
p_o + j\delta \cdot F \subset P_o.
\] 
\end{theorem}

\begin{remark}
Theorem \ref{Thm_R} corresponds to 
\[
d = m = 1 \qand \Lambda = \Delta = \Delta(\bR,\bZ[\sqrt{D}],[-1,1]).
\]
Furthermore, since $\Delta_n^n \subset \Delta$, we have $j\delta \in \Delta$ for all $j =1,\ldots,n$. In particular, this shows that the set $S_F$ in Theorem \ref{Thm_R} is non-empty.
\end{remark}

\subsection*{Ergodic $\IP$-theory}

Let $H$ be an abelian semi-group and $n$ a positive integer. We denote by $2^{[n]}$
the set of all subsets of the set $[n] = \{1,\ldots,n\}$. A function 
$\varphi : 2^{[n]} \ra H$ such that
\[
\varphi(\alpha \cup \beta) = \varphi(\alpha) + \varphi(\beta), \quad 
\textrm{for all $\alpha, \beta \in 2^{[n]}$ such that $\alpha \cap \beta = \emptyset$}
\]
is called an \emph{$\IP_n$-system with values in $H$}. 

\begin{example}
Let $L \subset H$ be a set with at most $n$ elements, and fix $h_1,\ldots,h_n \in L$ (allowing repetitions). Then 
\[
\varphi : 2^{[n]} \ra H, \quad \varphi(\alpha) = \sum_{k \in \alpha} h_{k}
\]
is an $\IP_n$-system with values in $H$. In particular, if we fix $h \in H$, and
set $h_k = h$ for all $k=1,\ldots,n$, then
\[
\varphi(\alpha) = |\alpha| \cdot h, \quad \textrm{for all $\alpha \in 2^{[n]}$}.
\]
We denote this special $\IP_n$-system by $\varphi_{n,h}$.
\end{example}

We will deduce Theorem \ref{Thm_Rext} from the following result (\cite[Theorem 9.5]{FK2}), which is only stated for subsets in $\bR^d$ in \cite{FK2}. We denote by $\bR^{+}$
the additive semigroup of positive real numbers.

\begin{theorem}[Furstenberg-Katznelson]
\label{ThmFK}
Let $A \subset \bR^d$ be a Borel set and suppose that 
\[
\eps := \varlimsup_{n \ra \infty} \frac{m_{\bR^d}(A \cap [-n,n]^d)}{n^d} > 0,
\]
Then, for every finite set $F \subset \bR^d$, there exists an integer $n = n(\eps,|F|)$ with the property that for every 
$\IP_n$-system $\varphi$ with values in $\bR^{+}$, there exist $\alpha \in 2^{[n]}$ and $x \in \bR^d$ such that 
\[
x + \varphi(\alpha) \cdot F \subset A.
\]
\end{theorem}

\subsection*{Proof of Theorem \ref{Thm_Rext}}

Let $\Lambda, \Delta, \Gamma$ and $W$ be as in Theorem \ref{Thm_Rext}. Let 
$F_o = \{\lambda_1,\ldots,\lambda_r\} \subset \Lambda^\infty$ be a finite set, 
and define $F = \{0,\lambda_1,\ldots,\lambda_r\}$. Let $q \geq 1$ be the smallest
integer such that $F \subset \Lambda^q$. Since $\Lambda$ is an approximate lattice,
there is an open identity neighbourhood $U$ in $\bR^d$ such that $\Lambda^{q+2} \cap U = \{0\}$. \\

Suppose $P_o \subset \Lambda$ has positive upper asymptotic density along $([-n,n]^d)$ in $\bR^d$. Fix an open identity neighbourhood $V$ in $\bR^d$ such that $V-V \subset U$, and define the open set $A = P_o + V \subset \bR^d$. 
Since $P_o$ has positive upper asymptotic density along $([-n,n]^d)$, we have
\[
\eps := \varlimsup_{n \ra \infty} \frac{m_{\bR^d}(A \cap [-n,n]^d)}{n^d} > 0.
\]
By Theorem \ref{ThmFK}, there exists an integer $n$, only depending on $\eps$ and $r$
with the property that for every $\IP_n$-system $\varphi$ with values in $\bR^{+}$, 
there exist $\alpha \in 2^{[n]}$ and $x \in \bR^d$ such that
\begin{equation}
\label{conclFK}
x + \varphi(\alpha) \cdot F \subset A = P_o + V.
\end{equation}
Define $\Delta_n = \Delta(\bR,\bR^m,\Gamma,\frac{1}{n}W^o)$, where $W^o \subset W$ denotes
the interior of $W$, which is assumed to be convex and symmetric. In particular, $\Delta_n$ is again a cut-and-project set, and is thus both relatively dense and uniformly discrete in $\bR^d$. Furthermore, 
\[
\Delta_n^n \subset \Delta\Big(\bR,\bR^m,\Gamma,\frac{1}{n}W^o + \ldots + \frac{1}{n}W^o\Big) \subset \Delta(\bR^m,\Gamma,W^o) \subset \Delta.
\]
Hence, if we pick any $\delta \in \Delta_n \cap \bR^{+}$, then 
$\varphi_{n,\delta} : 2^{[n]} \ra \bR^{+}$, as defined in the example above, is 
an $\IP_n$-system such that 
\[
\varphi_{n,\delta}(\alpha) = |\alpha| \cdot \delta \in \Delta, \quad \textrm{for all $\alpha \in 2^{[n]}$}.
\]
Fix $\delta \in \Delta_n$. By \eqref{conclFK}, 
there exist $j=1,\ldots,n$ and $x \in \bR^d$ such that
\[
x + j \delta \cdot \{0,\lambda_1,\ldots,\lambda_r\} \subset P_o + V.
\]
Set $\lambda_o = 0$. The inclusion above implies that there are $p_o,p_1,\ldots,p_r \in P_o$ and $v_o,v_1,\ldots,v_r \in V$ such that
\[
x + j\delta \cdot \lambda_k = p_k + v_k, \quad \textrm{for all $k=0,1,\ldots,r$}.
\]
Note that since $j\delta \in \Delta$ for all $j = 1,\ldots,n$ and $j\delta \cdot \Lambda \subset \Lambda$, we have
\[
j\delta \cdot \lambda_k \in j\delta \cdot \Lambda^q \subset \Lambda^q, \quad \textrm{for all $k=0,1,\ldots,r$}.
\]
By taking differences, 
\[
(p_k + v_k) - (p_o + v_o) = j\delta \cdot \lambda_k \in \Lambda^q, \quad \textrm{for all $k$},
\]
and thus $v_k-v_o \in (\Lambda^q - P_o + P_o) \cap (V-V) \subset \Lambda^{q+2} \cap U = \{0\}$ for all $k$. We conclude that $v_o = v_1 = \ldots = v_r$, so 
\[
p_k = p_o + j\delta \cdot \lambda_k, \quad \textrm{for all $k=1,\ldots,r$}.
\]
Hence, $p_o + j\delta \cdot \{\lambda_1,\ldots,\lambda_r\} \subset P_o$ and $p_o \in 
P_o$, which finishes the proof.

\subsection*{A few comments about the proof}

This proof can be extended to other lcsc abelian groups, thus providing \emph{some} (potentially weaker) version of our main combinatorial Theorem \ref{Thm_MainComb} for \emph{cut-and-project sets}. To also establish syndeticity of the sets $S$ in Theorem \ref{Thm_MainComb} along these lines (at least in the case when the ambient group $G$
contains a dense finitely generated group), a theory of $\IP^*$-sets in approximate lattices needs to be developed. However, the proof above uses that the approximate lattice $\Delta \subset \bR$ is a cut-and-project set, and thus contains, for every $n \geq 1$, a syndetic set $\Delta_n \subset \Delta$ such that $\Delta_n^n \subset \Delta$. As we have already pointed out in Remark \ref{Rmk_NoPwrs}, this property does not hold for arbitrary approximate lattices, so if one wants to prove Theorem \ref{Thm_MainComb} in its full generality using this approach, then some other way of producing $\IP$-systems inside an approximate lattice is needed.

\end{document}